\documentclass[12pt,reqno]{amsart}
\usepackage{epsf, amsmath,amsbsy,amsfonts, amsthm, amssymb,upref,yfonts}
\usepackage{amsrefs}

\usepackage{multirow}
\usepackage{hhline}
\usepackage[utf8]{inputenc}

\usepackage{mathtools}
\usepackage{times}
\usepackage[T1]{fontenc}
\usepackage{mathrsfs}
\usepackage{latexsym}
\usepackage[dvips]{graphics}
\usepackage[titletoc, title]{appendix}
\usepackage{amsmath,amsfonts,amsthm,amssymb,amscd}
\usepackage[dvipsnames]{xcolor}
\usepackage{hyperref}
\usepackage{amsmath}
\usepackage{IEEEtrantools}

\usepackage{color}
\usepackage{breakurl}

\usepackage{comment}
\newcommand{\bburl}[1]{\textcolor{blue}{\url{#1}}}

\newtheorem{thm}{Theorem}[section]

\newtheorem{cor}[thm]{Corollary}

\newtheorem{lem}[thm]{Lemma}
\newtheorem{prop}[thm]{Proposition}

\newtheorem{defi}[thm]{Definition}
\newtheorem{rek}[thm]{Remark}

\DeclareMathOperator{\supp}{supp}

\numberwithin{equation}{section}

\newcommand{\Max}{\text{\rm MAX}}

\DeclareFontFamily{U}{mathx}{}
\DeclareFontShape{U}{mathx}{m}{n}{<-> mathx10}{}
\DeclareSymbolFont{mathx}{U}{mathx}{m}{n}
\DeclareMathAccent{\widehat}{0}{mathx}{"70}
\DeclareMathAccent{\widecheck}{0}{mathx}{"71}

\begin{document}

\title[Higher order Tsirelson Spaces  and their modified version]{Higher order Tsirelson Spaces  and their modified versions  are isomorphic}
\author{H\`ung Vi\d{\^e}t Chu}
\address{Department of Mathematics, Texas A\&M University, College
  Station, TX 77843, USA}
  \email{hungchu1@tamu.edu}
\author{Thomas Schlumprecht}
\address{Department of Mathematics, Texas A\&M University, College
  Station, TX 77843, USA and Faculty of Electrical Engineering, Czech
  Technical University in Prague, Zikova 4, 166 27, Prague, Czech
  Republic}
\email{t-schlumprecht@tamu.edu}
\thanks{The work was supported in part by  the grant DMS-2054443 from the National Science Foundation.}


\begin{abstract}
We prove that for every countable ordinal $\xi$, the Tsirelson's space $T_\xi$ of order  $\xi$, is 
naturally, i.e.,  via the identity, $3$-isomorphc to its modified version. For the first step, 
we prove that the Schreier family $\mathcal{S}_\xi$ is the same as its modified version $
\mathcal{S}^M_\xi$, thus answering a question by Argyros and Tolias. 
 As an application, we show that the algebra of linear bounded operators on $T_\xi$ has $2^{\mathfrak c}$ closed ideals.
\end{abstract}

\maketitle

\tableofcontents
\section{Introduction}

In 1974, Tsirelson  \cite{Tsirelson74}  constructed the first known space that does not contain any isomorphic copies of $\ell_p$, $1\leqslant p<\infty$, or $c_0$. Soon later,
Figiel and Johnson \cite{FigielJohnson74} introduced  the  dual of Tsirelson's original space, and showed that the  $p$-convexification of it  is a uniformly convex Banach space, not containing  any $\ell_q$,
$1<q<\infty$. Nowadays the space defined by Figiel and Johnson is called {\em Tsirelson space}, and denoted by $T$, while 
Tsirelson's original space is often denoted by $T^*$.

In 1976, Johnson  \cite{Johnson76} defined a modified version $T^M$  of $T$, nowadays  called the {\em modified Tsirelson space}, which also does not contain  $\ell_p$, $1\leqslant p<\infty$, or $c_0$, but has the additional property that for any $m\in \mathbb N$, there is a co-finite dimensional subspace $Y$ in which every $m$-dimensional subspace is 2-isomorphic to a subspace of a finite-dimensional subspace 
$F$ of $Y$ which is $2$-isomorphic to $\ell_1^{\mbox{dim}(F)}$.
Later,  Casazza and Odell \cite{CasazzaOdell82/83} showed that the norms on  $T^M$ and $T$ are actually equivalent. 

The space $T$ is a reflexive space with a $1$-unconditional basis $(e_i)$, which has the property that the averages of any normalized block basis are not converging in norm to zero. It is therefore a reflexive example of a space in which a weakly null sequence may not have  Ces\'aro summing subsequences. In \cite{AlspachArgyros92}, Alspach and Argyros introduced {\em iterated averages 
of order $\xi$}, $\xi$ being a countable ordinal, and constructed reflexive spaces $T_\xi$ for which 
 the iterated averages of order $\xi$ do not converge to $0$. We refer to these spaces as {\em higher order Tsirelson spaces}.
As in the base case, we can also define their modified version. The main goal of this paper is to show that 
Tsirelson spaces of any order $\xi<\omega_1$ are isomorphic (via the identity) to their modified version. This extends 
 a result by Manoussakis  \cite{Manoussakis04} who proved this result for finite ordinals $\xi$.

Our paper is organized as follows: in Section 2, we recall Schreier families  $\mathcal{S}_\xi$ of order $\xi<\omega_1$, as well as their modified version $\mathcal{S}^M_\xi$,  $\omega_1$ denoting the first uncountable ordinal. The main goal of Section 3, is to show the equality of 
  $\mathcal{S}_\xi$ and  $\mathcal{S}^M_\xi$.
  In Section 4, we recall the definition of  Tsirelson spaces  $T_\xi$ of order  $\xi<\omega_1$,  and their modified versions $T^M_\xi$, and it is shown that their norms are equivalent.
   Section 5 is concerned with the cardinality of the closed ideals of the algebra of bounded 
   linear operators on $T_\xi$.

 \section{Schreier families of higher order}\
In this section, we recall  the definition of  \textit{Schreier families}    of order $\xi$, where $\xi$ is a countable ordinal, and their modified version. 
We introduce some notation: $[\mathbb{N}]$,   $[\mathbb{N}]^{<\omega}$,  and $[\mathbb{N}]^{\omega}$ denote the subsets, the finite subsets, and the infinite subsets of the natural numbers $\mathbb{N}$.
For $A\in [\mathbb{N}]^{<\omega}$ and $B\in[\mathbb{N}]$, we write $A < B$, if $a< b$ for all $a\in A$ and $b\in B$. 
For $A\in [\mathbb{N}]$ and $k\in \mathbb{N}$, we let 
$A_{>k}=A\cap\{k+1,k+2,\ldots\}$, $A_{\geqslant k}=A\cap\{k,k+1,k+2,\ldots\}$, $A_{<k}=A\cap\{1,2,\ldots, k-1\}$, and 
$A_{\leqslant k}=A\cap\{1,2,\ldots, k\}$.
  As a matter of convention, we put $\max \emptyset = 0$ and $\min \emptyset = \infty$,
  and thus, $A < \emptyset$ and $A > \emptyset$ for all $A \in [\mathbb{N}]^{<\omega}$.
  For $m\in\mathbb{N}$, we write $m\leqslant A$ (or $m<A$, respectively), if 
 $m\leqslant \min A$ (or $m<\min A$, respectively).
 
 Identifying  elements of  $[\mathbb{N}]$ in the canonical way with $\{0,1\}$-valued functions on $\mathbb{N}$ and thus elements of  $\{0,1\}^\mathbb{N}$, we equip $[\mathbb{N}]$ with the product
  topology of the discrete topology on $\{0,1\}$.
 We say that $\mathcal{F}\subset[\mathbb{N}]^{<\omega}$  is  \textit{compact} if it is compact in that product
topology. 
We say that $\mathcal{F}\subset[\mathbb{N}]^{<\omega}$  is  \textit{hereditary} if $A\in \mathcal{F}$ and $B\subset A$ imply that $B\in\mathcal{F}$.
Given $a_1 < \cdots < a_n, b_1 < \cdots < b_n$ in $\mathbb{N}$, we say that $\{b_1,\dots,b_n\}$
is \textit{a spread} of $\{a_1,\ldots,a_n\}$ if $a_i\leqslant b_i$ for all $i \in\{1,\ldots,n\}$. A family $\mathcal{F}\subset[\mathbb{N}]^{<\omega}$ is called
\textit{spreading} if every spread of every element of $\mathcal{F}$ is also in
$\mathcal{F}$. A family that is compact, hereditary, and spreading is said to be \textit{regular}.

We now define the Schreier family $\mathcal{S}_\xi\subset [\mathbb{N}]^{<\omega}$ for each $\xi < \omega_1$ by transfinite induction. Particularly,
$$\mathcal{S}_0\ :=\ \{\{n\}\,:\, n\in \mathbb{N}\}\cup \{\emptyset\}.$$
Assuming that $\mathcal{S}_\gamma$ has been defined for all $\gamma < \xi$. If $\xi = \gamma + 1$, we let 
\begin{equation}\label{e00}\mathcal{S}_\xi \ =\ \left\{\bigcup_{i=1}^n E_i\,:\, n\leqslant E_1 < E_2< \cdots < E_n, E_i\in \mathcal{S}_\gamma, i = 1, 2, \ldots, n\right\}.\end{equation}
If $\xi$ is a limit ordinal, we choose a fixed sequence of successor ordinals $(\alpha(\xi, n))_{n=1}^\infty\subset [1, \xi)$, which increases to $\xi$ and put
\begin{equation}\label{e12}\mathcal{S}_{\xi}\ = \ \{E: \exists n\ \leqslant\ E\mbox{ with }E\in \mathcal{S}_{\alpha(\xi, n)}\}.\end{equation}
The sequence $(\alpha(\xi, n))_{n=1}^\infty$ is called a $\xi$-approximating sequence. As needed later in Section \ref{S:6},
it follows by transfinite induction that $\mathcal{S}_1\subset \mathcal{S}_\xi$ for all $\xi<\omega_1$. Furthermore, we can and will later put additional conditions on $(\alpha(\xi, n))_{n=1}^\infty$.

The standard transfinite induction gives that the families $\mathcal{S}_\xi$, for $\xi < \omega_1$, are regular, and they are almost increasing in the following sense: for all ordinals $\gamma < \xi$,
\begin{equation}\label{e10}
\mbox{ there is an }m\in \mathbb{N}\mbox{ so that if }A\in \mathcal{S}_\gamma\mbox{ and }\min A\geqslant m,\mbox{ then }A\in \mathcal{S}_\xi. 
\end{equation}
In the case of a limit ordinal $\xi$, Property \eqref{e10} enables us to require the following  additional property on a $\xi$-approximating sequence $(\alpha(\xi, n))_{n=1}^\infty$:
\begin{align}\label{e11}
&\alpha(\xi,1)=1,\notag\\
 &\text{and for all  }n\in \mathbb{N},\text{ if }A, B\in \mathcal{S}_{\alpha(\xi, n)}\mbox{ with }1 < A < B,\mbox{ then }A\cup B \in \mathcal{S}_{\alpha(\xi, n+1)}.
\end{align}
Indeed, suppose that we have a limit ordinal $\xi$ and have chosen approximating sequences $(\alpha(\gamma,n))_{n=1}^\infty$ for all limit ordinals $\gamma<\xi$. We first choose an arbitrary  $\xi$-approximating sequence $(\alpha'(\xi, n))_{n=1}^\infty$. We shall build a new $\xi$-approximating sequence $(\alpha(\xi, n))_{n=1}^\infty$ from $(\alpha'(\xi, n))_{n=1}^\infty$
be choosing $\ell_1<\ell_2<\ell_3<\cdots $ recursively so that 
$$\{ A\cup B: A,B\in \mathcal{S}_{\alpha'(\xi,n)+\ell_n}, 1<A<B\} \ \subset\ \mathcal{S}_{\alpha'(\xi,n+1)+\ell_{n+1}},$$ and then putting $\alpha(\xi,n)=\alpha'(\xi,n)+\ell_1$.

Set $\ell_1 = 0$. Assume that $\ell_n$ has been defined. Consider $\alpha'(\xi, n+1) + \ell_n >  \alpha'(\xi, n) + \ell_n =\alpha(\xi, n) $. By \eqref{e10}, there exists $k_n\in \mathbb{N}$ such that $A_{\geqslant k_n}\in \mathcal{S}_{\alpha'(\xi, n+1) + \ell_n}$, for any $A\in \mathcal{S}_{\alpha(\xi, n)}$.
 Thus for any two sets $1 < A < B$ in $\mathcal{S}_{\alpha(\xi, n)} = \mathcal{S}_{\alpha'(\xi, n)+\ell_n}$
it follows that  $U := A_{\geqslant k_n}\cup B_{\geqslant k_n}\in \mathcal{S}_{\alpha'(\xi, n+1) + \ell_n+1}$. By \eqref{e00}, $\{2,3,\ldots, k_n-1\}\in\mathcal{S}_{\alpha'(\xi, n+1) + k_n-3}$; hence, 
$V:=A_{<k_n}\cup B_{<k_n}\in  \mathcal{S}_{\alpha'(\xi, n+1) + k_n-3}$. Since $V < U$ are in $\mathcal{S}_{\alpha'(\xi, n+1) + \max\{k_n-3, \ell_n+1\}}$, we know that $U\cup V\in \mathcal{S}_{\alpha'(\xi, n+1) + \max\{k_n-2, \ell_n+2\}}$. Therefore, if we set $\ell_{n+1} = \max\{k_n-2, \ell_n+1\}+1$ and $\alpha(\xi, n+1) = \alpha'(\xi, n+1) + \ell_{n+1}$. The new $\xi$-approximating sequence $(\alpha(\xi, n))_{n=1}^\infty$ increases to $\xi$ and satisfies $\eqref{e11}$.

As a consequence of \eqref{e11}, for a $\xi$-approximating sequence $(\alpha(\xi, n))_{n=1}^\infty$, we have 
${\mathcal S}_{\alpha(\xi, m)}\subset{\mathcal S}_{\alpha(\xi, n)}$ whenever $m\leqslant n$, which, in combination with \eqref{e12}, gives a neat way to write $\mathcal{S}_{\xi}$ as
\begin{equation}\label{e14}\mathcal{S}_{\xi}\ =\ \{E\,:\, \exists n\leqslant \min E \mbox{ with }E\in \mathcal{S}_{\alpha(\xi, n)}\}\ =\ \{E\,:\,E\in \mathcal{S}_{\alpha(\xi, \min E)}\}.\end{equation}

For $\xi < \omega_1$, we denote by MAX$(\mathcal{S}_\xi)$ the maximal elements of $\mathcal{S}_\xi$. Since $\mathcal{S}_\xi$ is compact, every element of $\mathcal{S}_\xi$ is a subset of some maximal element of $\mathcal{S}_\xi$ (see \cite[Section 2]{BrenchFerencziTcaciuc20} for details). 

\begin{prop}
If $0 < \xi = \gamma + 1 < \omega_1$ and $A\in\Max(\mathcal{S}_\xi)$, then there are unique elements $A_i\in \Max(\mathcal{S}_\gamma), i = 1, 2, \ldots, \min A$, so that 
$$A_1 \ <\ A_2 \ <\ \cdots\ <\ A_{\min A}\mbox{ and }A\ =\ \cup_{i=1}^{\min A}A_i.$$
\end{prop}
\begin{proof}
Pick $A\in \Max(\mathcal{S}_\xi)$. By definition, we can write $A = \cup_{i=1}^n A_i$, for $n\leqslant A_1 < \cdots < A_n$, and $A_i$ are nonempty sets in $\mathcal{S}_\gamma$. We shall show that $n = \min A$, and each $A_i$ is in $\Max(\mathcal{S}_\gamma)$.

Observe that $\min A_1 = \min A$, so $n\leqslant \min A$. If $n < \min A$, then find some $m > A_n$ to have 
$$n+1\ \leqslant\ A_1 < \cdots < A_n < A_{n+1}\ :=\ \{m\}$$
By definition, the set $A' = \cup_{i=1}^{n+1} A_i\in \mathcal{S}_\xi$ and $A\subsetneq A'$, which contradicts the maximality of $A$. Hence, $n = \min A$.

We claim that every $A_i$ is in $\Max(\mathcal{S}_\gamma)$. Otherwise, let $j$ be the smallest such that $A_j\notin \Max(\mathcal{S}_\gamma)$. If $j = \min A$, then form $A'_{\min A} := A_{\min A}\cup \{m\}\in \mathcal{S}_\gamma$, for some $m > A$. Then the set $\cup_{i=1}^{j-1}A_i \cup A'_{\min A}$ is in $\mathcal{S}_{\xi}$, again contradicting the maximality of $A$. If $j < \min A$, we form 
\begin{align*}
&A'_j\ :=\ A_j\cup \{\min A_{j+1}\}\\
&A'_i \ :=\ (A_i\backslash \{\min A_i\})\cup \{\min A_{i+1}\}\mbox{ for }j\ <\ i\ \leqslant\ \min A-1,\mbox{ and }\\
&A'_{\min A}\ :=\ (A_{\min A}\backslash \{\min A_{\min A}\})\cup \{m\}, \mbox{ for some }m > A.
\end{align*}
Clearly, $A'_i\in \mathcal{S}_\gamma$ for $i = j, \ldots, \min A$. Hence,
$$\left(\bigcup_{i=1}^{j-1} A_i\right)\cup\left(\bigcup_{i=j}^{\min A} A'_i\right)\ \in\ \mathcal{S}_\xi,$$
contradicting the maximality of $A$.

Finally, the uniqueness of the decomposition is obvious. Indeed, suppose that $A = \cup_{i=1}^{\min A} A_i = \cup_{i=1}^{\min A} A'_i$, where $A_1 < \cdots < A_{\min A}\in \Max(\mathcal{S}_\gamma)$ and $A'_1 < \cdots < A'_{\min A}\in \Max(\mathcal{S}_\gamma)$. Let $j$ be the smallest (if exists) such that $A_j\neq A'_j$. Then either $A_j\subsetneq A'_j$ or $A'_j\subsetneq A_j$. Either way violates that both $A_j$ and $A_j'$ are in MAX$(\mathcal{S}_\gamma)$. Therefore, $A_j = A_j'$ for $j = 1, \ldots, \min A$. 
\end{proof}

\begin{prop}
If $\xi < \omega_1$ is a limit ordinal, then $A\in \Max(\mathcal{S}_\xi)$ if and only if $A\in \Max(\mathcal{S}_{\alpha(\xi,\min A)})$.
\end{prop}
\begin{proof}
Let $A\in \Max(\mathcal{S}_\xi)$. By \eqref{e14}, $A\in \mathcal{S}_{\alpha(\xi, \min A)}$. If $A\notin \Max(\mathcal{S}_{\alpha(\xi, \min A)})$, then form $A' := A\cup \{\max A + 1\}$, which must still be in $\mathcal{S}_{\alpha(\xi, \min A)} = \mathcal{S}_{\alpha(\xi, \min A')}$. Hence, $A'\in \mathcal{S}_\xi$, contradicting the maximality of $A$. 

Conversely, let $A\in \Max(\mathcal{S}_{\alpha(\xi,\min A)})$. Then $A\in \mathcal{S}_\xi$. If $A\notin \Max(\mathcal{S}_\xi)$, then $A' := A\cup \{\max A + 1\}\in \mathcal{S}_\xi$. By the definition of $\mathcal{S}_\xi$, $A'\in \mathcal{S}_{\alpha(\xi, \min A')} = \mathcal{S}_{\alpha(\xi, \min A)}$, thus contradicting the maximality of $A$. 
\end{proof}

It is convenient to  define a $\xi$-approximating sequence $(\alpha(\xi, n))_{n=1}^\infty$ also for a successor ordinal $\xi < \omega_1$ by simply letting $\alpha(\xi, n) = \xi$ for all $n\in \mathbb{N}$. Thus, it follows that $A\in \mathcal{S}_{\alpha(\xi, \min A)}$ for $A\in \mathcal{S}_\xi$ (whether $\xi$ is a limit ordinal or not), and we have the following unified recursive definition of $\mathcal{S}_{\xi}$:
$$\mathcal{S}_{\xi}\ =\ \left\{\bigcup_{i=1}^{\min A_1} A_i\,:\, A_1 < \cdots < A_{\min A_1}\mbox{ and }A_i\in \mathcal{S}_{\alpha(\xi,\min A_1)-1}, i = 1, \ldots, \min A_1\right\},$$
where $\gamma-1$ is the predecessor for a sucessor ordinal $\gamma$. Moreover, if $A\in \Max(\mathcal{S}_\xi)$, then this decomposition into increasing sets $A_i\in \Max(\mathcal{S}_{\alpha(\xi,\min A_1)-1})$ is unique.

We conclude this section by introducing the \textit{modified Schreier families}, which are defined using transfinite induction for all $\xi < \omega_1$, as follows
$$\mathcal{S}_0^M \ =\ \mathcal{S}_0\ =\ \{\{n\}\,:\, n\in \mathbb{N}\}\cup \{\emptyset\}.$$
Assuming that $\mathcal{S}_\gamma^M$ has been defined for all $\gamma < \xi$. If $\xi = \gamma + 1$, we let
$$\mathcal{S}_\xi^M= \left\{\bigcup_{i=1}^n E_i\,:\, \mbox{sets }E_i\mbox{'s are pairwise disjoint, } E_i\in \mathcal{S}_\gamma^M, \mbox{ and }E_i\geqslant n, i = 1,\ldots, n\right\}.$$
If $\xi$ is a limit ordinal, we choose a fixed sequence $(\alpha(\xi, n):n\in \mathbb{N})\subset [1,\xi)$ which increases to $\varepsilon$ and put
$$\mathcal{S}^M_{\xi} = \{E\,:\, \exists n\ \leqslant\ \min E \mbox{ with }E\in \mathcal{S}^M_{\alpha(\xi, n)}\}.$$

\section{Equality of the families $\mathcal{S}_\xi$ and $\mathcal{S}^M_\xi$, $\xi<\omega_1$}
In this section, we answer a question stated in \cite[pp.\ 16]{ArgyrosTolias04} by showing that for any $\xi < \omega_1$, the Schreier family $\mathcal{S}_\xi$ of order $\xi < \omega_1$ coincides with its modified version $\mathcal{S}^M_\xi$. This result was independently obtained in \cite{CauseyPelczar-Barwacz23}.
 We would like to give here a different, self-contained proof and lay the foundation for the next section, where we prove that the spaces  $T_\xi$ and $T^M_\xi$ are naturally (via the formal identity) isomorphic. More precisely, we will prove Theorem \ref{m1}, from which, it will easily follow that $\mathcal{S}_\xi= \mathcal{S}_\xi^M$, for $\xi<\omega_1$. Theorem \ref{m1} will also be used in Section \ref{iso}.

\begin{defi}[The $\xi$-analysis of maximal $\mathcal{S}_\xi$ sets]\normalfont

Given $\xi < \omega_1$, we define recursively the $\xi$-analysis $T(\xi, A)$ of a maximal $\mathcal{S}_\xi$ set $A$ or \textit{the $\xi$-analysis tree of $A$}, for short.  We also define for any $D\in T(\xi, A)$, the \textit{$\xi$-order of $D$ inside $A$}, denoted by $\mbox{order}_{\xi, A}(D)$, which  is an ordinal not larger than $\xi$. 

If $\xi = 0$ and $A = \{a\} \in \Max(\mathcal{S}_0)$,  put $T(0, A) = \{A\}$ and order$_{\xi, A}(A) = 0$.
Assume that $T(\gamma, A)$ has been defined for all $\gamma < \xi$, and  $A\in \Max(\mathcal{S}_\gamma)$, and for all $D\in T(\gamma, A)$, we have defined the ordinal $\mbox{order}_{\gamma, A}(D)\leqslant \gamma$. We define $T(\xi, A)$ for $A\in \mathcal{S}_\xi$.

\noindent{\bf Case 1:} $\xi$ is a successor ordinal. Write $\xi = \gamma + 1$. For $A\in \Max(\mathcal{S}_\xi)$, write $A = \cup_{i=1}^{\min A}A_i$, where $A_1 < \cdots < A_{\min A}$ and $A_i\in \Max(\mathcal{S}_\gamma)$. Then 
$$T(\xi, A)\ =\ \{A\}\cup \bigcup_{i=1}^{\min A}T(\gamma, A_i),\mbox{ and }$$
$$\mbox{order}_{\xi, A}(A) \ =\ \xi, \mbox{ and }\mbox{order}_{\xi, A}(D)\ =\ \mbox{order}_{\gamma, A_i}(D),$$
for $D\in T(\gamma, A_i)$ and $i = 1, \ldots, \min A$. 

\noindent{\bf Case 2:} $\xi$ is a limit ordinal. For $A\in \Max(\mathcal{S}_\xi)$, let $T(\xi, A) = T(\alpha(\xi, \min A), A)$, and $\mbox{order}_{\xi, A}(D) = \mbox{order}_{\alpha(\xi, \min A), A}(D)$ for all $D\in T(\xi, A)$.  
\end{defi}

\begin{rek}\normalfont
For $\xi<\omega_1$ and $A\in \Max(\mathcal{S}_\xi)$, we make the following observations:
\begin{enumerate}
\item If $\xi$ is a successor ordinal or $0$, we have $\mbox{order}_{\xi, A}(A) = \xi$; if $\xi$ is a limit ordinal, we have $\mbox{order}_{\xi, A}(A) = \alpha(\xi, \min A)$, and $T(\xi, A) = T(\alpha(\xi, \min A), A)$.
\item For every $D \in T(\xi, A)$, $\mbox{order}_{\xi, A}(D)$ is a successor ordinal, and furthermore, $D\in \Max(\mathcal{S}_{\mbox{order}_{\xi, A}(D)})$. 
\item If $D, E\in T(\xi, A)$ and $D\subsetneq E$, then $\mbox{order}_{\xi, A}(D) < \mbox{order}_{\xi, A}(E)$.
\item $T(\xi, A)$ is a tree with respect to inclusion, i.e., $D$ is a successor of $E$ if $D\subsetneq E$. In this tree, $A$ is the root, while the leaves (or terminal nodes) consist of all the singletons $\{a\}$, for $a\in A$. 
\item Let $D\in T(\xi, A)$ be not a terminal node. Denote the set of immediate successors of $D$ in $T(\xi, A)$ by $\mbox{succ}_{\xi}(D)$, i.e., 
$\mbox{succ}_\xi (D) =$ 
$$\{E\in T(\xi, A)\,:\, E\subsetneq D\mbox{ and }\not\exists E'\in T(\xi, A)\mbox{ s.t. }E\subsetneq E'\subsetneq D\}.$$
Then $D = \cup_{E\in \mbox{succ}(D)} E$. If we write $\mbox{order}_{\xi, A}(D) = \gamma + 1$, then $\mbox{succ}(D)$ consists of $\min D$ elements $D_1 < D_2 < \cdots < D_{\min D}$, where $D_i\in \Max(\mathcal{S}_\gamma)$. Then there are two cases: 
\begin{itemize}
\item either $\gamma$ is a successor ordinal, in which case $\mbox{order}_{\xi, A}(D_i) = \gamma$.
\item or $\gamma$ is a limit ordinal, in which case, we have $D_i\in \Max(\mathcal{S}_{\alpha(\gamma, \min D_i)})$ and $\mbox{order}_{\xi, A}(D_i)\ =\ \alpha(\gamma, \min D_i)$ for $i=1,2,\ldots, \min D$.
\end{itemize}
\item For $D\in T(\xi, A)$, $T(\mbox{order}_{\xi, A}(D), D)$ is a subtree of $A$, and
$$T(\mbox{order}_{\xi, A}(D), D)\ =\ \{E\in T(\xi, A)\,:\, E\subset D\}.$$
\item Let $D, E\in T(\xi, A)$, $D\neq E$, then there are only four possibilities:
\begin{itemize}
\item $D\subsetneq E$, and order$_{\xi, A}(D) < \mbox{order}_{\xi, A}(E)$,
\item $E\subsetneq D$, and order$_{\xi, A}(E) < \mbox{order}_{\xi, A}(D)$,
\item $E < D$,
\item $D < E$.
\end{itemize}
\end{enumerate} 
\end{rek}

We define for $\xi < \omega_1$, $A\in \Max(\mathcal{S}_\xi)$, and $a\in A$, \textit{the $\xi$-order of $a$ inside $A$}:
$$\mbox{order}_{\xi, A}(a)\ =\ \max\{\mbox{order}_{\xi, A}(D)\,:\, D\in T(\xi, A)\mbox{ and }a\ =\ \min D\}.$$
Roughly speaking, to find $\mbox{order}_{\xi, A}(a)$, we look for the set $D$ closest to the root and whose minimum is $a$, then set $\mbox{order}_{\xi, A}(a) =  \mbox{order}_{\xi, A}(D)$. Note that for every $a\in A$, we have $\{a\}\in T(\xi, A)$, so $\mbox{order}_{\xi, A}(a)$ is well-defined, and $\mbox{order}_{\xi, A}(a) = 0$ if and only if $a$ is not the minimum of any other node of $T(\xi, A)$. For $a\in A$, let $D_{\xi, A}(a)$ be the unique set $D$ closest to the root (possibly the root itself) and whose minimum is $a$. If $E\in T(\xi, A)$ and $a\in E$ with $a > \min E$, then $D_{\xi, A}(a) = D_{\mbox{order}_{\xi, A}(E),E}(a)$ and $\mbox{order}_{\xi, A}(a) = \mbox{order}_{\mbox{order}_{\xi, A}(E),E}(a)$.

\begin{lem}[The replacement lemma]\label{repl} Let $\xi < \omega_1$, $A\in \Max(\mathcal{S}_\xi)$, and $a\in A$ with $a > \min A$.  Put 
\begin{align*}
m_1 &\ =\ \max\{\max D\,:\, D\in T(\xi, A), D < D_{\xi, A}(a)\},\mbox{ and }\\
m_2 &\ =\ \min\{\min D\,:\, D\in T(\xi, A), D > D_{\xi, A}(a)\},
\end{align*}
where $\min \emptyset = \infty$ in the case there is no node after $D_{\xi, A}(a)$. Let $E$ be the immediate predecessor of $D_{\xi, A}(a)$ in $T(\xi, A)$. Write $\mbox{order}_{\xi, A}(E) = \gamma + 1$. Then for every $\widetilde{D}\in \Max(\mathcal{S}_\gamma)$ with $\widetilde{D}\subset (m_1, m_2)$, we have
$$\widetilde{A}\ :=\ (A\backslash D_{\xi, A}(a))\cup \widetilde{D}\ \in\ \Max(\mathcal{S}_\xi).$$
\end{lem}

\begin{proof}
We prove the lemma by induction on $n$, the number of predecessors of $D_{\xi, A}(a)$.

\noindent Base case: $n = 1$. Let $\xi < \omega_1$, $A\in \Max(\mathcal{S}_{\xi})$, and $a\in A$ with $a > \min A$. 
Since $n = 1$, $A$ is the immediate predecessor of $D_{\xi, A}(a)$ in $T(\xi, A)$. Write $A = \cup_{i=1}^{\min A} A_i$, where $\min A \leqslant A_1 < A_2 < \cdots < A_{\min A}$ and $A_i\in \Max(\mathcal{S}_\gamma)$. For some $i_0 = 2, 3, \ldots, \min A$, we have $A_{i_0} = D_{\xi, A}(a)$. Note that $i_0\neq 1$ because $a > \min A$. We have
$$A_1 \ <\ A_2 \ <\ \cdots \ <\ A_{i_0-1} \ <\ \widetilde{D}\ <\ A_{i_0+1}\ <\ \cdots \ <\ A_{\min A}.$$
It may be possible that in the above chain, $\widetilde{D}$ is the last set, which happens exactly when $i_0 = \min A$. By definition, 
$$\widetilde{A} \ =\ (A\backslash D_{\xi, A}(a))\cup\widetilde{D}\ =\ \bigcup_{i=1, i\neq i_0}^{\min A}A_i\cup \widetilde{D}\ \in\ \Max(\mathcal{S}_\xi).$$

\noindent Inductive hypothesis: assume that the lemma is true for $n$ predecessors and all $\xi < \omega_1$. 

Let us fix $\xi < \omega_1$, $A\in \Max(\mathcal{S}_{\xi})$, and $a\in A$ with $a > \min A$. Suppose that $D_{\xi, A}(a)$ has $n+1$ predecessors $(E_i)_{i=1}^{n+1}$ in $T(\xi, A)$ with
$$D_{\xi, A}(a)\ =\ E_0\ \subsetneq\ E_1\ \subsetneq\ E_2\ \subsetneq\ \cdots\ \subsetneq\ E_n\ \subsetneq\ E_{n+1} \ =\ A.$$
Put 
\begin{align*}
m_1 &\ =\ \max\{\max D\,:\, D\in T(\xi, A), D < D_{\xi, A}(a)\},\mbox{ and }\\
m_2 &\ =\ \min\{\min D\,:\, D\in T(\xi, A), D > D_{\xi, A}(a)\}.
\end{align*}
Choose $\widetilde{D}\subset (m_1, m_2)$ such that $\widetilde{D}\in \Max(\mathcal{S}_{\mbox{order}_{\xi, A}(a)})$.

Let $\xi = \gamma+1$ and consider the tree $T(\gamma, E_n)$. Put 
\begin{align*}
m'_1 &\ =\ \max\{\max D\,:\, D\in T(\gamma, E_n), D < D_{\gamma, E_n}(a)\},\mbox{ and }\\
m'_2 &\ =\ \min\{\min D\,:\, D\in T(\gamma, E_n), D > D_{\gamma, E_n}(a)\}.
\end{align*}
Since $D_{\gamma, E_n}(a) = D_{\xi, A}(a)$ and $T(\gamma, E_n)$ is a subtree of $T(\xi, A)$, we know that 
$$m'_1\ \leqslant\ m_1 < m_2 \ \leqslant\ m_2'.$$
Furthermore, $a > \min E_n$ because $D_{\xi, A}(a)$ has at least two predecessors. Applying the inductive hypothesis to the tree $T(\gamma, E_n)$ and $a\in E_n$, we have
$$(E_n\backslash D_{\xi, A}(a))\cup \widetilde{D}\ =\ (E_n\backslash D_{\gamma, E_n}(a))\cup \widetilde{D}\ \in\ \Max(\mathcal{S}_\gamma).$$
It is worth noting that due to how $m_1'$ and $\widetilde{D}$ are chosen, $(E_n\backslash D_{\xi, A}(a))\cup \widetilde{D}$ and $E_n$ have the same minimum. Write $A = \cup_{i=1}^{\min A}A_i$, for $\min A\leqslant A_1 < \cdots < A_{\min A}$ and $A_i\in \Max(\mathcal{S}_{\gamma})$. Let $A_{i_0} = E_n$. We have
$$A_1 \ <\ \cdots\ <\ A_{i_0-1} \ <\ (E_n\backslash D_{\xi, A}(a))\cup \widetilde{D}\ <\ A_{i_0+1}\ <\ \cdots < A_{\min A}.$$
These sets are in $\Max(\mathcal{S}_\gamma)$, and the transformation does not change the minimum of $A_{i_0}$. Therefore, $$(A\backslash D_{\xi, A}(a))\cup \widetilde{D} \ =\ \left(\cup_{i=1, i\neq i_0}^{\min A} A_i\right) \cup ((E_n\backslash D_{\xi, A}(a))\cup \widetilde{D})\ \in\ \Max(\mathcal{S}_\xi).$$
This completes our proof. 
\end{proof}

\begin{thm}\label{m1}
Let $\xi < \omega_1$. If $A_1, \ldots, A_n$, $n\in \mathbb{N}$, are pairwise disjoint, nonempty sets in $\mathcal{S}_\xi$ with 
$\min A_1 < \cdots < \min A_n$, then there exist nonempty sets $A_1' < \cdots < A_n'$ in $\mathcal{S}_\xi$ such that 
$\cup_{i=1}^n A_i = \cup_{i=1}^n A_i'$ and $\min A_i\leqslant \min A_i'$, for $i = 1, \ldots, n$. 
\end{thm}

\begin{rek}\normalfont
Let us first note that if $\xi < \omega_1$ and $A\in \mathcal{S}_\xi$ with $1\in A$, then $A = \{1\}$. Therefore, we can assume in the statement of Theorem \ref{m1} that $1\notin A_i$ for $i = 1, 2, \ldots, n$. 
\end{rek}

\begin{proof}[Proof of Theorem \ref{m1}]
We will prove the theorem by transfinite induction for all $\xi < \omega_1$. For the base case $\xi = 0$, the theorem is trivially true. 

Inductive hypothesis: suppose that the theorem is true for all $\eta < \xi$, for some $\xi\geqslant 1$. We shall show that the theorem is true for the ordinal $\xi$. We do so in three steps:
\begin{itemize}
\item Step 1: we prove the theorem for a successor ordinal $\xi$ and $n = 2$.
\item Step 2: we prove the theorem for a limit ordinal $\xi$  and $n = 2$.
\item Step 3: we show that the theorem holds for an arbitrary $n\in \mathbb{N}$.
\end{itemize}

\noindent{\bf Step 1:} Write $\xi = \eta + 1$. Let $A, B\in \mathcal{S}_\xi$ be nonempty. Without loss of generality, we can assume that $2\leqslant \min A < \min B < \max A$. By enlarging $A$, if necessary, we also assume that $A\in \Max(\mathcal{S}_{\xi})$. 

We decompose $A$ as $A = \cup_{i=1}^{\min A} A_i$ with $A_i\in \mathcal{S}_\eta\backslash \{\emptyset\}$ and $\min A\leqslant A_1 < \cdots < A_{\min A}$. Also decompose $B = \cup_{i=1}^\ell B_i$ with $B_i\in \mathcal{S}_\eta\backslash \{\emptyset\}$ and $\ell \leqslant B_1 < \cdots < B_\ell$. Then we order the sets $\{A_i: i \leqslant \min A\}\cup \{B_i: i\leqslant \ell\}$ into sets $(D_i)_{i=1}^{\min A+\ell}$ and apply the inductive hypothesis to obtain nonempty sets $(\widetilde{D}_i)_{i=1}^{\min A+\ell}\subset \mathcal{S}_\eta$ such that
$$\widetilde{D}_1 \ <\ \cdots \ <\ \widetilde{D}_{\min A+\ell}\mbox{ and }\bigcup_{i=1}^{\min A+\ell}D_i \ =\ \bigcup_{i=1}^{\min A+\ell}\widetilde{D}_i.$$

We form $D'_1$ by unioning $\widetilde{D}_1$ with the smallest elements of $\cup_{i=2}^{\min A+\ell} \widetilde{D}_i$, if necessary, until we obtain a maximal $\mathcal{S}_\eta$ set. Then form $D'_2$ from the smallest elements in $(\widetilde{D}_i)_{i=2}^{\min A+\ell}\backslash D'_1$ so that $D'_2$ is a maximal $\mathcal{S}_\eta$ set. In general, we form $D'_j$, for $j = 2, \ldots, \min A$, from the smallest elements in $(\widetilde{D}_i)_{i=j}^{\min A+\ell}\backslash \left(\cup_{i=1}^{j-1} D'_i\right)$. Then the sets $(D'_i)_{i=1}^{\min A}$ are maximal $\mathcal{S}_\eta$. Let $\widetilde{A} = \cup_{i=1}^{\min A} D'_i$.

After the above process, some of the sets $(\widetilde{D}_i\backslash \widetilde{A})_{i=\min A+1}^{\min A+\ell}$ are possibly empty. Discard the empty sets and rename the nonempty ones as $(D'_i)_{i=\min A+1}^{\min A + \ell'}$. Here $\ell'\leqslant \ell$, and there must exist some nonempty sets; otherwise, we have a contradiction with the maximality of $A$. Let $\widetilde{B} = \cup_{i=\min A+1}^{\min A + \ell'} D'_i$.

Since all the sets $(B_i)_{i=1}^\ell$ have $\min B_i\geqslant \min B_1$, there are at least $\ell$ sets $\widetilde{D}_i$'s whose minimum is at least $\min B_1$. Hence, 
we have
\begin{align*}
&\min \widetilde{A} \ =\ \min D'_1\ =\ \min \widetilde{D}_1\ = \ \min A,\\
&\min \widetilde{B} \ =\ \min D'_{\min A+1}\ \geqslant\ \min \widetilde{D}_{\min A + 1}\ \geqslant\ \min B_1\ \geqslant\ \ell\ \geqslant\ \ell'. 
\end{align*}
Hence, both $\widetilde{A}$ and $\widetilde{B}$ are in $\mathcal{S}_\xi$, $\widetilde{A} < \widetilde{B}$, and $\min A = \min \widetilde{A}$ and $\min B\leqslant \min \widetilde{B}$, as desired. Moreover, if $A$ is $\mathcal{S}_\xi$ maximal, then $\widetilde{A}$ is $\mathcal{S}_\xi$ maximal as well. 

\noindent{\bf Step 2:} Let $\xi$ be a limit ordinal and $n = 2$. Let $A, B\in \mathcal{S}_\xi$ be nonempty. Without loss of generality, we can assume that $\min A < \min B < \max A$ and $A\in \Max(\mathcal{S}_{\xi})$. Let $a = \min\{x\in A: x>\min B\}$ and let $\gamma+1 = \mbox{order}_{\xi, A}(E)$, where $E$ is the immediate predecessor of $D_{\xi, A}(a)$ in the tree $T(\xi, A)$. 

Put $F = \{x\in A\cup B: x\geqslant \min B\}$ and write $F = \{f_1, f_2, \ldots, f_\ell\}$, with $\min B = f_1 < f_2 < \cdots < f_\ell$. Then $D_{\xi, A}(a)\cup B\subset F$. Let $m < \ell$ be such that $\widetilde{D} = \{f_1, f_2, \ldots, f_m\}\in \Max(\mathcal{S}_\gamma)$. Such an $m$ must exist because $F$ contains a maximal $\mathcal{S}_\gamma$ set, namely $D_{\xi, A}(a)$. Both $\widetilde{D}$ and $D_{\xi, A}(a) = \{f_1', f_2', \ldots, f_k'\}$ are subsets of $F$, and because $\widetilde{D}$ consists of the smallest elements of $F$, we know that $f_i < f_i'$ for $i = 1, \ldots, \min \{m, k\}$. It must be that $k \geqslant m$; otherwise, the spreading property and $\widetilde{D}\in \mathcal{S}_\gamma$ would imply that $D_{\xi, A}(a)\cup \{a'\}\in \mathcal{S}_\gamma$, for every $a' > \max A\cup B$, is in $\Max(\mathcal{S}_\gamma)$. This contradicts the maximality of $D_{\xi, A}(a)$. Therefore, $\max \widetilde{D} = f_m < f'_k = \max D_{\xi, A}(a)$. We obtain
$$\max \widetilde{D}\ <\ \max D_{\xi, A}(a)\ <\ \min \{\min D\,:\, D\in T(\xi, A)\mbox{ and }D > D_{\xi, A}(a)\}\ =\ m_2.$$
Furthermore, 
$$\min \widetilde{D}\ =\ \min B\ >\ \max\{\max D\,:\, D\in T(\xi, A)\mbox{ and }D < D_{\xi, A}(a)\}\ =\ m_1.$$
It follows from Lemma \ref{repl} that $\widetilde{A} = (A\backslash D_{\xi, A}(a))\cup \widetilde{D}\in \Max(\mathcal{S}_\xi)$, and $\min A = \min \widetilde{A}$. Define $\widetilde{B} = (B\backslash \widetilde{D})\cup (D_{\xi, A}(a)\backslash \widetilde{D})$. It is easy to verify that
\begin{align*}
&\widetilde{A}\cup \widetilde{B}\ =\ A\cup B\mbox{ and }\min B \ <\ \min \widetilde{B},\\
&B\backslash \widetilde{D}\ \in\ \mathcal{S}_{\alpha(\xi, \min B)},\mbox{ and}\\
&D_{\xi, A}(a)\backslash \widetilde{D}\ \in\ \mathcal{S}_{\alpha(\xi, \min A)}\ \subset\ \mathcal{S}_{\alpha(\xi, \min B)}.
\end{align*}
Applying the inductive hypothesis to $\widetilde{B}$, we can find two sets $\widetilde{B}_1, \widetilde{B}_2\in \mathcal{S}_{\alpha(\xi, \min B)}$ with $\widetilde{B}_1 < \widetilde{B}_2$. By \eqref{e11}, we know that $\widetilde{B} = \widetilde{B}_1\cup \widetilde{B}_2\in \mathcal{S}_{\alpha(\xi, \min B+1)}\subset \mathcal{S}_{\alpha(\xi, \min \widetilde{B})}$. Hence, $\widetilde{B}\in \mathcal{S}_{\xi}$.

Either we are now in the situation that $\widetilde{A} < \widetilde{B}$ in which case, we are done, or $\max \widetilde{A} > \min \widetilde{B}$, and we would repeat the above process. However, we can only iterate the process finitely many times because, after each iteration, we increase the minimum of a subset of $A\cup B$ by $1$. This completes our Step 2. As in Step 1, our new set $\widetilde{A}$ is can be chosen to be in $\Max(\mathcal{S}_\xi)$ if $A\in \Max(\mathcal{S}_\xi)$. 

\noindent{\bf Step 3:} We show by induction on $n$ that if $A_1, A_2, \ldots, A_n$ are disjoint, nonempty sets in $\mathcal{S}_\xi$, then there are nonempty sets $\widetilde{A}_1, \widetilde{A}_2, \ldots, \widetilde{A}_n$ in $\mathcal{S}_\xi$ so that $\widetilde{A}_1 < \widetilde{A}_2 < \cdots < \widetilde{A}_n$ and $\cup_{j=1}^n A_j = \cup_{j=1}^n \widetilde{A}_j$. Assume that $\min A_1 < \min A_2 < \cdots < \min A_{n+1}$. By possibly enlarging the sets, we can also assume that $A_1, A_2, \ldots, A_{n+1}\in \Max(\mathcal{S}_\xi)$.

We can recursively find $n$ pairs of nonempty sets $(\widetilde{A}_1^{(j)}, A_j')\in \Max(\mathcal{S}_\xi)\times \mathcal{S}_\xi$, where $j = 2, \ldots, n+1$, so that
\begin{align}
&\widetilde{A}_1^{(j)}\ <\ A_i',\quad i = 2, \ldots, j, \mbox{ and}\label{cc1}\\
&\widetilde{A}_1^{(j)}\cup \bigcup_{i=2}^{j}A_i'\ =\ \bigcup_{i=1}^j A_i\label{cc2}.
\end{align}
Indeed, the case $j = 2$ follows from Step 1 and Step 2. Suppose we have chosen $\widetilde{A}_1^{(j)}$ and $A'_i$ for $j\leqslant n$ and $i = 2, \ldots, j$. From $\widetilde{A}_1^{(j)}$ and $A_{j+1}$, we obtain $\widetilde{A}_1^{(j+1)}\in \Max(\mathcal{S}_\xi)$ and $A'_{j+1}\in \mathcal{S}_\xi$ so that $\widetilde{A}_1^{(j+1)} < A'_{j+1}$ and $\widetilde{A}_1^{(j+1)}\cup A'_{j+1} = \widetilde{A}_1^{(j)}\cup A_{j+1}$. Note that $\max \widetilde{A}_1^{(j+1)} \leqslant \max \widetilde{A}_1^{(j)}$ because otherwise, $\widetilde{A}_1^{(j)}\subsetneq \widetilde{A}_1^{(j+1)}\in \mathcal{S}_\xi$, which contradicts the $\mathcal{S}_\xi$-maximality of $\widetilde{A}_1^{(j)}$.
This together with the supposition that \eqref{cc1} holds for $j$ guarantees that \eqref{cc1} holds for $j+1$. Using $\widetilde{A}_1^{(j+1)}\cup A'_{j+1} = \widetilde{A}_1^{(j)}\cup A_{j+1}$ and the supposition that \eqref{cc2} holds for $j$, we can easily verify that \eqref{cc2} holds for $j+1$. This shows the existence of $\widetilde{A}_1^{(j)}$ and $A_j'$, $j = 2, 3, \ldots, n+1$.

Put $\widetilde{A}_1 = \widetilde{A}_1^{(n+1)}$. Then $\widetilde{A}_1 < A_i'$, for $i = 1, 2, \ldots, n+1$. Apply the inductive hypothesis to the $n$ sets $(A_i')_{i=2}^{n+1}\subset \mathcal{S}_\xi$ to have $(\widetilde{A}_i)_{i=2}^{n+1}\subset\mathcal{S}_\xi$ satisfying
$$\widetilde{A}_2 \ <\ \cdots \ <\ \widetilde{A}_{n+1}\quad \mbox{ and }\quad \bigcup_{i=2}^{n+1} \widetilde{A}_{i}\ =\ \bigcup_{i=2}^{n+1} A'_i.$$
It follows from \eqref{cc2} that 
$$\widetilde{A}_1\ <\  \widetilde{A}_2 \ <\ \cdots \ <\ \widetilde{A}_{n+1}\quad \mbox{ and }\quad \bigcup_{i=1}^{n+1} \widetilde{A}_{i}\ =\ \widetilde{A}_1^{(n+1)}\cup \bigcup_{i=2}^{n+1} A'_i\ =\ \bigcup_{i=1}^{n+1} A_i.$$
This completes our proof. 
\end{proof}

\begin{cor}\label{mc1}
For all countable ordinal $\xi\in  [1,\omega_1)$, we have $\mathcal{S}_\xi = \mathcal{S}^M_\xi$. 
\end{cor}

\begin{proof}
We proceed by transfinite induction.
The case $\xi = 1$ is trivial. Assume that the corollary holds for all $\gamma < \xi$ for some $\xi\geqslant 2$. 

\noindent{\bf Case 1:} if $\xi$ is a successor ordinal, then write $\xi = \gamma + 1$. Let $F\in \mathcal{S}^M_\xi$. By definition, $F$ can be written as $\cup_{i=1}^k F_i$, for $(F_i)_{i=1}^k\subset \mathcal{S}^M_\gamma$, being pairwise disjoint, and $(\min F_i)_{i=1}^k\in \mathcal{S}_1$. By the inductive hypothesis, $\mathcal{S}^M_\gamma = \mathcal{S}_\gamma$; hence, $(F_i)_{i=1}^k\subset \mathcal{S}_\gamma$. 
By Theorem \ref{m1}, we can find $(F_i')_{i=1}^k\subset \mathcal{S}_\gamma$ such that
$$F'_1 \ <\ F'_2 \ <\ \cdots \ <\ F'_k, \{\min F'_i\,:\, 1\leqslant i\leqslant k\}\in \mathcal{S}_1, \mbox{ and }\cup_{i=1}^k F'_i \ =\ \cup_{i=1}^k F_i.$$
Therefore, $F$ is in $\mathcal{S}_\xi$. 

\noindent{\bf Case 2:} $\xi$ is a limit ordinal. Pick $F\in \mathcal{S}^M_\xi$, then $F\in \mathcal{S}^M_{\alpha(\xi, k)}$ for some $k\leqslant \min F$. By the inductive hypothesis, $F\in \mathcal{S}_{\alpha(\xi, k)}\subset \mathcal{S}_{\alpha(\xi, \min F)}$, which implies that $F\in \mathcal{S}_\xi$. 
\end{proof}

\section{Modified Tsirelson spaces of higher order}\label{iso}

Manoussakis \cite{Manoussakis04} proved that the Tsirelson space $T[\mathcal{S}_\xi, \theta]$ is naturally isomorphic to its modified version $T_M[\mathcal{S}_\xi, \theta]$ for all finite ordinals $\xi$. To do so, he compared $T[\mathcal{S}_\xi, \theta]$ to an auxiliary space $Y$ then compared $T_M[\mathcal{S}_\xi, \theta]$ to $Y$, where the latter step used a technique from Bellenot \cite{Bellenot86}. In this section, we extended Manoussakis' result by proving the isomorphism between $T[\mathcal{S}_\xi, \theta]$ and $T_M[\mathcal{S}_\xi, \theta]$ for all ordinals $\xi < \omega_1$. To compare  $T[\mathcal{S}_\xi, \theta]$ to $Y$, we borrow Manoussakis' argument and adapt it to handle limit ordinals (see Lemma \ref{l10}.) To compare $T_M[\mathcal{S}_\xi, \theta]$ to $Y$, we present a different proof than Manoussakis' and Bellenot's.

Let $c_{00}$ be the vector space of all sequences $x=(x_n)$  in $\mathbb R$ for which 
the support, $\text{supp } x=\{n\in \mathbb N: x_n\not= 0\}$, is finite. We denote by $(e_n)$ the unit vector basis of $c_{00}$ and by $(e^*_n)$ its coordinate functionals.
All our Banach spaces will be completions of $c_{00}$ under some norm, for which $(e_n)$ is a normalized and unconditional basis. 

Let $\mathcal{F}$ denote a regular family of finite subsets of $\mathbb{N}$. A sequence of finite sets $(E_i)_{i=1}^k$ of positive integers is said to be $\mathcal{F}$-{\em admissible} if 
$$E_1 \ <\ E_2 \ <\ \cdots \ <\ E_k\mbox{ and }\{\min E_i\,:\, 1\leqslant i\leqslant k\}\ \in\ \mathcal{F}.$$
A sequence of finite sets $(E_i)_{i=1}^k$ of positive integers is said to be $\mathcal{F}$-{\em allowable} if 
$$(E_i)_{i=1}^k\mbox{ are pairwise disjoint and }\{\min E_i\,:\, 1\leqslant i\leqslant k\}\ \in\ \mathcal{F}.$$
Let $\theta$ be a number with $0 < \theta < 1$.
The Tsirelson space $T[\mathcal{F}, \theta]$ is the completion of $c_{00}$ under the implicit norm
$$\|x\|_{\mathcal{F}, \theta}\ =\ \max\left\{\|x\|_\infty, \theta\sup\left\{\sum_{i=1}^k \|E_ix\|\,:\, (E_i)_{i=1}^k \mbox{ is }\mathcal{F}\mbox{-admissible}\right\}\right\}.$$
The modified Tsirelson space $T_M[\mathcal{F}, \theta]$ is the completion of $c_{00}$ under the implicit norm
$$\|x\|_{\mathcal{F}, \theta, M} \ =\ \max\left\{\|x\|_\infty, \theta\sup\left\{\sum_{i=1}^k \|E_ix\|\,:\, (E_i)_{i=1}^k \mbox{ is }\mathcal{F}\mbox{-allowable}\right\}\right\}.$$
It was shown in \cite{AlspachArgyros92}, where the space $T[\mathcal{S}_\xi, \theta]$ was introduced, that 
 $T[\mathcal{S}_\xi, \theta]$ is reflexive, and, thus, in particular, the unit basis  
 cannot have a block basis equivalent to the unit vector basis
 of $\ell_1$. It follows from \cite[Theorem 1.27]{ArgyrosDeliyanniKutzarovaManoussakis98} that  $T_M[\mathcal{S}_\xi,\theta] $ is also reflexive.

Our goal in this section is to show that $T[\mathcal{S}_\xi, \theta]$ is $3$-isomorphic to $T_M[\mathcal{S}_\xi, \theta]$ for every countable ordinal $\xi$ and for every $0<\theta < 1$. To do so, we shall use the auxiliary space $T[\mathcal{S}_\xi[\mathcal{A}_2], \theta]$, where $\mathcal{A}_n = \{E\subset \mathbb{N}: |E|\leqslant n\}$. It suffices to prove the following two theorems. 
\begin{thm}\label{mm1}
For every $x\in c_{00}$, we have
$$\|x\|_{\mathcal{S}_\xi[\mathcal{A}_2],\theta}\ \leqslant\ 3\|x\|_{\mathcal{S}_\xi, \theta},$$
\end{thm}

\begin{thm}\label{mm2}
For every $x\in c_{00}$, we have
$$\|x\|_{\mathcal{S}_\xi, \theta, M}\ \leqslant\ \|x\|_{\mathcal{S}_\xi[\mathcal{A}_2],\theta}$$
and thus, by Theorem \ref{mm1},
$$\|x\|_{\mathcal{S}_\xi, \theta}\ \leqslant\ \|x\|_{\mathcal{S}_\xi, \theta,M}\ \leqslant\ 3\|x\|_{\mathcal{S}_\xi, \theta}.$$
\end{thm}

To prove Theorem \ref{mm1}, we recall some notation for the statement of the key Lemma \ref{l10}. 
Given two families $\mathcal{M}$ and $\mathcal{N}$ of subsets of $\mathbb{N}$, let 
$$\mathcal{M}[\mathcal{N}]\ =\ \left\{\bigcup_{i=1}^k F_i\,:\, F_i\in \mathcal{N}\mbox{ and }(F_i)_{i=1}^k\mbox{ is }\mathcal{M}\mbox{-admissible}\right\}, \mbox{ and }$$
$$(\mathcal{M})^3\ =\ \mathcal{A}_3[\mathcal{M}]\ =\ \left\{M_1\cup M_2\cup M_3\,:\, M_1 < M_2 < M_3\mbox{ and }M_i\in \mathcal{M}\right\}.$$

Below are several useful facts:
\begin{enumerate}
    \item For families $\mathcal{L}, \mathcal{M}$, and $\mathcal{N}$, it holds that
    \begin{equation}\label{e31}
    (\mathcal{L}[\mathcal{M}])[\mathcal{N}]\ =\ \mathcal{L}[\mathcal{M}[\mathcal{N}]].
    \end{equation}
    \begin{proof}
    Let $F_1 < \cdots < F_k$, where $F_i\in \mathcal{N}$ and $(\min F_i)_{i=1}^k\in \mathcal{L}[\mathcal{M}]$. Setting $j_0 = 0$, we can find $\ell$ and  $j_1< \cdots< j_\ell$ such that $\{\min F_{j_{i-1}+1}, \ldots, \min F_{j_{i}}\}\in \mathcal{M}$ for $i=1, \ldots, \ell$, and $(\min F_{j_{i}+1})_{i=0}^{\ell-1}\in \mathcal{L}$. Now let $E_n = \cup_{i=j_{n-1}+1}^{j_{n}} F_{i}$, $n = 1, 2, \ldots, \ell$. 
    Clearly, $E_n\in \mathcal{M}[\mathcal{N}]$. Therefore, $\cup_{i=1}^k F_i = \cup_{n=1}^\ell E_n\in  \mathcal{L}[\mathcal{M}[\mathcal{N}]]$.
    
     Conversely, we take $F_1 < \cdots < F_k$, where each $F_i\in \mathcal{M}[\mathcal{N}]$ and the set $\{\min F_1, \ldots, \min F_k\}$ is in $\mathcal{L}$. Write $F_i = \cup_{j=1}^{\ell_i} E^i_j$, for $E^i_j\in \mathcal{N}$, $(\min E^i_j)_{j=1}^{\ell_i}\in \mathcal{M}$, and $E^i_1 < \cdots < E^i_{\ell_i}$. Note that $\{\min E^i_j: 1\leqslant i\leqslant k, 1\leqslant j\leqslant \ell_i\}$ is in $\mathcal{L}[\mathcal{M}]$. Hence, $\cup_{i=1}^k\left(\cup_{j=1}^{\ell_i} E^i_j\right)\in  (\mathcal{L}[\mathcal{M}])[\mathcal{N}]$.

    This completes our proof. 
    \end{proof}
    \item From  \eqref{e31}, we deduce that for  a family $\mathcal{F}$ of subsets of $\mathbb{N}$ and  for all $\xi < \omega_1$, 
    \begin{equation}\label{e30}\mathcal{S}_{\xi+1}[\mathcal{F}]\ =\ \mathcal{S}_1[\mathcal{S}_\xi[\mathcal{F}]].\end{equation}
    \item For all $\xi < \omega_1$, 
    \begin{equation}\label{e32}\mathcal{S}_1[(\mathcal{S}_\xi)^3]\ \subset\ (\mathcal{S}_{\xi+1})^3.\end{equation}
    \begin{proof}
    Take $F_1  < \cdots < F_k$, $F_i\in (\mathcal{S}_\xi)^3$, and $\{\min F_1, \ldots, \min F_k\}\in \mathcal{S}_1$. Write $F_i = F^i_1\cup F^i_2\cup F^i_3$, for $F^i_j\in \mathcal{S}_\xi$ and $F^i_1 < F^i_2 < F^i_3$. Reindex the sets $\{F^i_j: 1\leqslant i\leqslant k, 1\leqslant j\leqslant 3\}$ into $(E_i)_{i=1}^{3k}$ satisfying
    $k\ \leqslant \ E_1 \ <\ E_2 \ <\ \cdots\ <\ E_{3k}$. 
    Since $k\leqslant E_1 < E_2 < \cdots < E_k$ and $E_i\in \mathcal{S}_\xi$, we know that $\cup_{i=1}^k E_i\in \mathcal{S}_{\xi+1}$.  Next, $2k\leqslant E_{k+1} < E_{k+2} < \cdots < E_{3k}$ and $E_i\in \mathcal{S}_\xi$; hence, $\cup_{i=k+1}^{3k} E_i\in \mathcal{S}_{\xi+1}$. We conclude that $$\bigcup_{i=1}^k F_i \ =\  \bigcup_{i=1}^k E_i\cup \bigcup_{i=k+1}^{3k} E_i\ \in\ (\mathcal{S}_{\xi+1})^2.$$
    \end{proof}
\end{enumerate}

Using the above facts, we can prove the following lemma. 

\begin{lem}\label{l10}
Let $1\leqslant \xi < \omega_1$. It holds that 
\begin{equation}\label{e33}(\mathcal{S}_\xi[\mathcal{A}_2])[\mathcal{A}_3] \ \subset\ (\mathcal{S}_\xi)^3.\end{equation}
\end{lem}

\begin{proof}
We proceed by transfinite induction on $\xi$. When $\xi = 1$, \eqref{e31} gives
$$(\mathcal{S}_1[\mathcal{A}_2])[\mathcal{A}_3] \ =\ \mathcal{S}_1[\mathcal{A}_2[\mathcal{A}_3]]\ =\ \mathcal{S}_1[\mathcal{A}_6].$$
Let $F \in \mathcal{S}_1[\mathcal{A}_6]$, then $F = \cup_{i=1}^k F_i$, where $k\leqslant F_1 < F_2 < \cdots < F_k$ and $|F_i|\leqslant 6$. Write $F = \{m_1, m_2, \ldots, m_\ell\}$ with $k\leqslant m_1\leqslant m_2 \leqslant \cdots \leqslant m_\ell$ and $\ell\leqslant 6k$. The same reasoning as in the proof of \eqref{e32} guarantees that $F\in (\mathcal{S}_1)^3$.

Inductive hypothesis: Suppose that \eqref{e33} holds for all $\gamma < \xi$ for some $\xi\geqslant 2$. We shall show that it is true for $\xi$.

\noindent{\bf Case 1:} $\xi$ is a successor ordinal. Write $\xi = \gamma+1$. By \eqref{e31}, the inductive hypothesis, and \eqref{e32}, we have
\begin{align*} (\mathcal{S}_{\gamma+1}[\mathcal{A}_2])[\mathcal{A}_3]& \ =\ ((\mathcal{S}_1[\mathcal{S}_\gamma])[\mathcal{A}_2])[\mathcal{A}_3] \ =\ (\mathcal{S}_1[\mathcal{S}_\gamma[\mathcal{A}_2]])[\mathcal{A}_3]\\
&\ = \ \mathcal{S}_1 [(\mathcal{S}_\gamma[\mathcal{A}_2]) [\mathcal{A}_3]]
\ \subset\ \mathcal{S}_1 [(\mathcal{S}_\gamma)^3]\ \subset\ (\mathcal{S}_{\xi})^3.
\end{align*}

\noindent{\bf Case 2}: $\xi$ is a limit ordinal. Take $F\in (\mathcal{S}_{\xi}[\mathcal{A}_2])[\mathcal{A}_3]$. Then $F = \cup_{i=1}^k F_i$, for $|F_i|\leqslant 3$ and $(\min F_i)_{i=1}^k\in \mathcal{S}_{\xi}[\mathcal{A}_2]$. Hence, $\{\min F_i: i = 1, 3, \ldots, 2\lfloor (k-1)/2\rfloor+1\}\in \mathcal{S}_{\xi}$. 
By \eqref{e14}, $\{\min F_i: i = 1, 3, \ldots, 2\lfloor (k-1)/2\rfloor+1\}\in \mathcal{S}_{\alpha(\xi, \min F)}$. By the inductive hypothesis, $F\in (\mathcal{S}_{\alpha(\xi, \min F)}[\mathcal{A}_2])[\mathcal{A}_3]\subset (\mathcal{S}_{\alpha(\xi, \min F)})^3$. Hence, $F = E_1\cup E_2 \cup E_3$, for $E_1 < E_2 < E_3$ and $E_i\in \mathcal{S}_{\alpha(\xi, \min F)}$. Since $E_i\in \mathcal{S}_{\alpha(\xi, \min F)}$ and $\min E_i\geqslant \min F$, it follows from \eqref{e14} that $E_i\in \mathcal{S}_{\xi}$, for $i = 1, 2, 3$. Therefore, $F\in (\mathcal{S}_{\xi})^3$.
\end{proof}

For a  regular family $\mathcal F$ and $0<\theta<1$, we now define  $1$-norming sets $K=K(\mathcal F,\theta)$ and $K^M=K^M(\mathcal{F}, \theta)$ for $T[\mathcal{F}, \theta]$ and $T_M[\mathcal{F}, \theta]$, respectively. Let $K_0=K_0(\mathcal{F},\theta) = \{0\}\cup \{\pm e^*_n: n\in \mathbb{N}\}$. Assuming that $K_n=K_n(\mathcal F)$ has been defined, we 
define $K_{n+1}=K_{n+1}(\mathcal F)$ by
$$K_{n+1}\ =\ K_n\cup \{\theta(f^*_1 + \cdots + f^*_d)\,:\, d\in \mathbb{N}, f_i^*\in K_n, \mbox{ and }(\supp f^*_i)_{i=1}^d\mbox{ is }\mathcal{F}\mbox{-admissible}\}.$$
Let $K = \cup_{n=0}^\infty K_n$. Then $K$ is a $1$-norming set for $T[\mathcal{F}, \theta]$, i.e., for every $x\in T[\mathcal{F}, \theta]$,
$$\|x\|_{\mathcal{F}, \theta}\ =\ \sup_{f^*\in K}\langle x, f^*\rangle.$$
Similarly, we  define a norming set $K^M$ for $T_M[\mathcal{F}, \theta]$. Let $K_0^M=K_0^M(\mathcal F) = \{0\}\cup \{\pm e^*_n: n\in \mathbb{N}\}$. Assuming that $K_n^M=K_n^M(\mathcal{F})$ has been defined, we define $K^M_{n+1}=K^M_{n+1}(\mathcal F)$ by
$$K^M_{n+1}\ =\ K^M_n\cup \{\theta(f^*_1 + \cdots + f^*_d)\,:\, d\in \mathbb{N}, f_i^*\in K^M_n, (\supp f^*_i)_{i=1}^d\mbox{ is }\mathcal{F}\mbox{-allowable}\}.$$
From now on, by the norming set of $T[\mathcal{F}, \theta]$ and $T_M[\mathcal{F}, \theta]$, we are referring to either $K$ and $K^M$, respectively.

We observe the following. 
Let $f^*\in K^M$, and we write it as $f^*=\sum_{j=1}^\infty a_j e^*_j$,  for some $(a_j)\in c_{00}$, and if $N\subset \mathbb N$
then 
$P_N(f^*)= \sum_{j\in N} a_j e^*_j\in K^M$; moreover, $P_N(f^*)\in K$ if $f^*\in K$.
We call $P_N(f^*)$  the {\em  projection of $f^*$ onto $N$}.

\begin{proof}[Proof of Theorem \ref{mm1}]
Let $x\in c_{00}$ and $f^*$ be arbitrary in the norming set $K = \cup_{n=0}^\infty K_n$ of $T[\mathcal{S}_{\xi}[\mathcal{A}_2], \theta]$. To verify our claim, it suffices to show by induction on $n$ that there exist $h^*_1, h^*_2$, and  $h^*_3$ in the norming set of $T[\mathcal{S}_\xi, \theta]$ such that
\begin{equation}\label{e34}f^*(x)\ \leqslant\ \sum_{i=1}^3 f^*_i(x), \bigcup_{i=1}^3 \supp f^*_i \subset \supp f^*, \mbox{ and } \supp f^*_1 \!<\! \supp f^*_2 \!<\! \supp f^*_3.\end{equation}
If $f^* = \pm e^*_m$ for some $m\in \mathbb{N}$, there is nothing to prove. Suppose that \eqref{e34} holds for every $f^*\in K_n$ for some $n\geqslant 0$. We show that it holds for $f^*\in K_{n+1}\backslash K_n$. Write $f^* = \theta \sum_{i=1}^d f^*_i$, for $f^*_i\in K_n$ and $(\supp f^*_i)_{i=1}^d$ is $\mathcal{S}_{\xi}[\mathcal{A}_2]$-admissible. By the inductive hypothesis, for $i = 1, 2, \ldots, d$, we can find three functionals $g^{*}_{i,1}, g^*_{i,2}$, and $g^*_{i,3}$ in the norming set of $T[\mathcal{S}_\xi, \theta]$ satisfying 
\begin{equation*}f_i^*(x)\ \leqslant\ \sum_{j=1}^3 g^*_{i,j}(x), \bigcup_{j=1}^3 \supp g^*_{i,j}\ \subset\ \supp f^*_i, \mbox{ and } \supp g^*_{i,1} \ <\ \supp g^*_{i,2} \ <\ \supp g^*_{i,3}.\end{equation*}
Lemma \ref{l10} and $(\min\supp f^*_i)_{i=1}^d\in \mathcal{S}_{\xi}[\mathcal{A}_2]$ imply that 
$$\{\min\supp g^*_{i, j}\,:\, 1\leqslant i\leqslant d, 1\leqslant j\leqslant 3\}\ \in\ (\mathcal{S}_{\xi}[\mathcal{A}_2])[\mathcal{A}_3]\ \subset\ (\mathcal{S}_{\xi})^3.$$
Reindex $(g^*_{i, j})_{i, j}$ into the sequence $(g^*_i)_{i=1}^{3d}$ with consecutive supports and find $j_1$ and $j_2$ such that 
\begin{align*}
&\{\min\supp g^*_1, \ldots, \min\supp g^*_{j_1}\}\ \in\ \mathcal{S}_\xi,\\
&\{\min\supp g^*_{j_1+1}, \ldots, \min\supp g^*_{j_2}\}\ \in\ \mathcal{S}_\xi,\mbox{ and }\\
&\{\min\supp g^*_{j_2+1}, \ldots, \min\supp g^*_{3d}\}\ \in\ \mathcal{S}_\xi.
\end{align*}
Then $h_1^* = \theta\sum_{i=1}^{j_1} g^*_i$, $h_2^* = \theta\sum_{i=j_1+1}^{j_2} g^*_i$, and $h_3^* = \theta\sum_{i=j_2+1}^{3d} g^*_i$ are in the norming set of $T[\mathcal{S}_\xi, \theta]$, satisfying $\supp h_1^* < \supp h_2^* < \supp h_3^*$ and $\cup_{i=1}^3 \supp h_i^*\subset \supp f^*$. Furthermore, 
$$
f^*(x)  \ =\ \theta\sum_{i=1}^d f^*_i(x)\ \leqslant\ \theta\sum_{i=1}^d \sum_{j=1}^3 g^*_{i, j}(x) \ =\ \theta\sum_{i=1}^{3d}g^*_i(x) \ =\ \sum_{i=1}^3 h_i^*(x).
$$
This completes our proof. 
\end{proof}

In order to prove Theorem \ref{mm2}, we shall verify
\begin{equation}\label{E:5.3}
K^M(\mathcal{S}_\xi, \theta)\ \subset\ K(\mathcal{S}_\xi[\mathcal{A}_2], \theta), \text{ for all   $\xi<\omega_1$,}
\end{equation}
which implies that
\begin{equation*}
\|x\|_{\mathcal{S}_\xi, \theta, M}\ \leqslant\ \|x\|_{\mathcal{S}_\xi[\mathcal{A}_2], \theta}, \text{ for all }x\in c_{00}.
\end{equation*}
Note that, for a  given  regular regular family $\mathcal F\subset [\mathbb N]^{<\omega}$, the sets
$K(\mathcal{F})$ and $K^M(\mathcal{F})$
are invariant under sign changes of the coefficients, i.e., for all   
$k\in \mathbb N$, $(a_i)_{i=1}^k\subset \mathbb R$, $n_1<n_2<\cdots <n_k$ in $\mathbb N$, and   
$(\varepsilon_j)_{j=1}^k\subset \{\pm1\}$, we have
\begin{align*}
\sum_{j=1}^k a_j e^*_{n_j} \ \in\  K(\mathcal F)&\iff \sum_{j=1}^k\varepsilon_j a_j e^*_{n_j} \ \in\  K(\mathcal F),\mbox{ and }\\
\sum_{j=1}^k a_j e^*_{n_j} \ \in\  K^M(\mathcal F)&\iff \sum_{j=1}^k\varepsilon_j a_j e^*_{n_j} \ \in\  K^M(\mathcal F).
\end{align*}
It is therefore enough to show that 
\begin{equation}\label{E:5.4}
P^M(\mathcal{S}_\xi, \theta)\ \subset\ P(\mathcal{S}_\xi[\mathcal{A}_2], \theta),
\end{equation}
where for a regular family $\mathcal F$, we denote by   $P^M(\mathcal F)$ and $P(\mathcal F)$ the elements 
of    $K^M(\mathcal F)$ and $K(\mathcal F)$, respectively, with nonnegative coefficients.

We fix a regular family $\mathcal F$ and abbreviate $P=P(\mathcal F)$ and   
$P^M=P^M(\mathcal F)$; for $n\in \mathbb N_0$, we let $P_n=P(\mathcal F) \cap K_n(\mathcal F)$ and  $P^M_n=P^M(\mathcal F) \cap K^M_n(\mathcal F)$.

\begin{enumerate}
\item If $f^* \in P_0^M= P_0\setminus\{0\}$, then $f^*=e^*_s$, for some $s\in \mathbb N$.
\item If
 $\ell\in \mathbb N$ and $f^*\in P^M_\ell\backslash P_0^M$,
then the  coefficients of $f^*$ are  of the form $\theta^j$, with  $j\in\{1,2,\ldots, \ell\}$.
We define the {\em  level sets } of $f^*$ to be
$$L_j(f^*)\ =\ \{ s\in\mathbb{N}\,:\,  f^*(e_s)=\theta^j\}, \text{ for $j=1,2,\ldots, \ell$.}$$
It follows that the  sets $L_j(f^*)$ are pairwise disjoint and that
$$f^*\ =\ \sum_{j=1}^\ell \theta^j \sum_{s\in L_j(f^*)} e^*_s.$$
\end{enumerate}

For $A\in [\mathbb N]^{<\omega}$, we  write $1^*_A$ for the functional $\sum_{s\in A} e^*_s$. We will now discuss conditions  under which   a pairwise disjoint sequence $(A_j)_{j=1}^\ell$ in $[\mathbb N]^{<\omega}$
has the property that $\sum_{j=1}^\ell\theta^j1^*_{A_j}\in P_\ell^M$.
If that is the case, and if $A_\ell\not=\emptyset$ (and thus, $\sum_{j=1}^\ell\theta^j1^*_{A_j}\in P_\ell^M\setminus P_{\ell-1}^M$),  we call $(A_j)_{j=1}^\ell$ an {\em $\mathcal F$-allowable sequence of length $\ell$}.

\begin{prop}\label{P:5.1}  Assume that $(A_j)_{j=1}^\ell$ is  an $\mathcal F$-allowable sequence of length $\ell$
and that 
\begin{equation}\label{eh1} A'_\ell\mbox{ is a spread of }A_\ell \mbox{ and is disjoint from }\bigcup_{j=1}^{\ell-1} A_j.\end{equation}
Then $(A_1,A_2,\ldots, A_{\ell-1}, A'_\ell)$ is also $\mathcal F$-allowable of length $\ell$.
\end{prop}
\begin{proof} The case for $\ell=1$ follows from the assumption that $\mathcal F$ is spreading.
Assume that the claim is true for some $\ell\in\mathbb N$, and let  $(A_j)_{j=1}^{\ell+1}$
be an $\mathcal F$-allowable sequence of length  $\ell+1$, and let $A'_{\ell+1}$ be a spread of $A_{\ell+1}$. 
We have
$f^*=\sum_{j=1}^{\ell+1} \theta^j1^*_{A_j}\in P_{\ell+1}^M\setminus P_\ell^M,$
and thus we can write  $f^*$ as $f^*=\theta \sum_{i=1}^k f^*_i$
with $k\in \mathbb N$, $f^*_i\in P_\ell^M$, for $i=1,2,\ldots, k$, and $(\mbox{supp}(f^*_i))_{i=1}^k$
being $\mathcal F$-allowable. It follows that 
\begin{align*}
A_1&\ =\ L_1(f^*)\ =\ \{ s\in\mathbb N\,:\, \exists i=1,2,\ldots, k \mbox{ with } f^*_i=e^*_s\},
\intertext{and for $j=2,3,\ldots, \ell+1$, we have}
A_j&\ =\ L_j(f^*)\ =\ \bigcup_{i=1}^k L_{j-1}(f^*_i).
\end{align*}
Since $A'_{\ell+1}$ is a spread of $A_{\ell+1}$, there is a bijection
$\rho:A_{\ell+1}\to A'_{\ell+1}$, with $\rho(a)\geqslant a$ for $a\in A_{\ell+1}$.
For $i=1,2,\ldots, k$, put  $A'_{\ell+1,i} = \rho(L_\ell(f^*_i))$, then 
$A'_{\ell+1,i}$ is  a spread of $L_\ell(f^*_i)$, or $L_\ell(f^*_i)=\emptyset$.
For $i = 1, \ldots, k$, let 
$$g^*_i\ = \ \begin{cases}f^*_i,&\mbox{ if }L_\ell(f^*_i)=\emptyset,\\ \sum_{j=1}^{\ell-1} \theta^j 1^*_{L_j(f^*_i)}+ \theta^\ell 1_{A'_{\ell+1,i}}^*, &\mbox{ otherwise}.\end{cases}$$
By the inductive hypothesis, we have $g^*_i\in P^M_\ell$.
Note that \eqref{eh1} and the fact that $(\supp(f^*_i))_{i=1}^k$ is $\mathcal{F}$-allowable implies that $(\mbox{supp}(g^*_i))_{i=1}^k$
is $\mathcal F$-allowable.
Therefore, we deduce that  
$$\sum_{j=1}^\ell \theta^j 1^*_{A_j}+ \theta^{\ell+1}1^*_{A'_{\ell+1}}\ =\ \theta \sum_{i=1}^k g^*_i\ \in\ P^M_{\ell+1}.$$
This completes our proof. 
\end{proof} 

We introduce a {\em coding }for elements in $P$ and $P^M$ using certain finite trees. For $n\in \mathbb N$, we let $[[\mathbb N]]^n$ be the set of non strictly increasing sequences in $\mathbb N$ of length $n$, and put 
 $[[\mathbb N]]^{<\omega}=\bigcup_{n=1}^\infty [[\mathbb N]]^n$. 
 For $s=(s_1,s_2, \ldots, s_m)$ and $t=(t_1,t_2,\ldots, t_n)$, we say that $t$ {\em extends  }$s$ and write 
 $s\prec t$ if $m<n$, $s_i=t_i$, for $i=1,2,\ldots,m$, and we write 
 $s \preceq t$ if  $s\prec t$ or $s=t$. For $t=(t_1,t_2,\ldots, t_n)=[[\mathbb N]]^{<\omega}$, we call $n$ the {\em length of $t$}, denoted by $|t|$;
 for $j=1,2,\ldots, |t|$, we denote the $j$th-element of $t$ by $t(j)$.
 If $s\in [[\mathbb N]]^{m} $ and $t\in [[\mathbb N]]^{n} $, for some $m,n\in \mathbb N$, and $s_m\leqslant t_1$,
 we put 
 $$(s,t)\ =\ (s_1,s_2,\ldots ,s_m, t_1,t_2,\ldots, t_n)\ \in\ [[\mathbb N]]^{m+n}.$$
 We also define $(\emptyset,s)=(s,\emptyset)=s$.

We introduce the {\em lexicographical order} on $[[\mathbb N]]^{<\omega}$:
for $s=(s_1,s_2,\ldots, s_m)$ and $t=(t_1,t_2,\ldots,t_n)$ in $ [[\mathbb N]]^{<\omega}$, we let 
$$s\ \prec_{\text{lex}}\ t\iff \exists 1\ \leqslant\ i_0 \ <\ \min(m,n)\mbox{ s.t. }  s_i\ =\ t_i\text{ for $1\leqslant i<i_0$ and } s_{i_0}\ <\ t_{i_0}.$$
Note that if $s,t\in[[\mathbb N]]^{<\omega}$ are not comparable with each other with respect to 
``$\preceq$'', then either $s\prec_\text{lex} t$ or $t\prec_\text{lex} s$.

A finite subset $T$ of $[[\mathbb N]]^{<\omega}$ is called an  $\mathcal F$-{\em allowable tree} if 
\begin{enumerate}
\item[(T1)] If $t\in T$ and $s\preceq t$, then $s\in T$.
\item[(T2)] There is a $t_0\in T$ with $|t_0|=1$, which has the property  that $t_0\preceq t$ for all $t\in T$.
\end{enumerate}
From (T1) and (T2), it follows that $T$ is a tree,  for which $t_0$ is the root.
For  $t\in T\setminus \{t_0\}$, we let
$t^-$  be the unique predecessor of $t$ in $T$.
For $t\in T$, we call $s\in T$ an immediate successor of $t$
if $t\prec s$ and $|s|=|t|+1$ and 
denote the set of immediate successors of $t$ by $S^T_t$.  A node
$t\in T$ is called a {\em terminal node} if $S_t^T=\emptyset$.
\begin{enumerate}
\item[(T3)] If $t\in T$ is not a terminal node, then $(t,t(n))=(t(1), t(2),\ldots, t(n),t(n))\in T$, 
and $\{j\in \mathbb N: (t,j)\in T\}\in \mathcal F$.
\end{enumerate}
To state the next condition, we define the {\em support of a node $s\in T$} as
$$\mbox{supp}(s)\ =\ \{ t(|t|)\, :\,  t\succeq s, \mbox{ and } t \text{ is terminal in $T$}\}.$$
\begin{enumerate}
\item[(T4)]  If $s,t\in T$ with $s\prec_{\text{lex}} t$, then $\supp(s)$ and $\supp(t)$ are disjoint.
\end{enumerate}
We call $T$ an {\em $\mathcal F$-admissible tree} if in addition to Conditions  (T1) -- (T3), the following condition holds
\begin{enumerate}
\item[(T5)]   If $s,t\in T$ with $s\prec_{\text{lex}} t$, then $\supp(s)<\supp(t)$.               
\end{enumerate}
For  an $\mathcal F$-allowable tree $T$, we call
$$|T|\ =\ \max\{ |t|\,:\, t\in T\} \ =\  \max\{ |t|\,:\, t \text{ is terminal node of $T$}\} $$
the {\em length of $T$}. Let $\mathcal T^{(\ell)}$ and $\mathcal T^{M, (\ell)}$ be the the collection of all $\mathcal F$-admissible and $\mathcal{F}$-allowable trees of length $\ell$, respectively.
We identify  $\mathcal{T}^{(1)} = \mathcal{T}^{M, (1)}$ with $\mathbb N$ and put $\mathcal T=\bigcup_{\ell=1}^\infty \mathcal{T}^{(\ell)}$ and $\mathcal T^M = \bigcup_{\ell=1}^\infty \mathcal{T}^{M, (\ell)}$.
The following easy proposition exhibits two ways to reduce trees to trees of smaller lengths.
\begin{prop}\label{P:5.2} Let $\ell\in \mathbb N_{\geqslant 2}$ and $T\in \mathcal{T}^{M, (\ell)}$.  
\begin{enumerate}
\item  $T'=\{t\in T\,:\, |t|\leqslant \ell-1\} $ is in $\mathcal{T}^{M, (\ell-1)}$.
\item Let $n_1$ be the root of $T$ and  let $n_1<n_2<\cdots< n_k$ be chosen in $\mathbb N$, so that 
$$\{(n_1,n_i)\,:\, i=1,2,\ldots, k\}\ =\ \{ t\in T\,:\, |t|=2\}.$$
Then for $i=1,2,\ldots, k$, 
$$T^{(i)}\ =\ \{ (n_i,s)\,:\, s\in \{\emptyset\}\cup [[\mathbb N]]^{<\omega} 
\text{ with } (n_1,n_i,s)\in T\}\ \in\ \bigcup_{j=1}^{\ell-1} \mathcal{T}^{M, (j)},$$
and at least for some $i=1,2,\ldots, k$, we have $T^{(i)}\in \mathcal{T}^{M, (\ell-1)}$.
\end{enumerate}
Moreover, if $T$ is in $\mathcal{T}^{(\ell)}$, then
$T'$ as defined in (1) is in $\mathcal{T}^{(\ell-1)}$, and $T^{(i)}$, $i=1,2,\ldots, k$, as defined in (2), are in $\cup_{j=1}^{\ell-1}\mathcal{T}^{(j)}$.
\end{prop}

We define now for each $\mathcal F$-allowable tree $T$ a functional $\Psi(T)$ as follows.

\begin{prop}\label{P:5.3}
For $T\in \mathcal T^M$, we define
$$
 \Psi(T)\ =\ \begin{cases}  e^*_n\quad  &\text{if $T=\{n\}\in \mathcal T^{M, (1)}$,}\\
  \sum_{j=2}^\ell \theta^{j-1} 1^*_{\{t(j): t\in T\text{is terminal node, and $|t|=j$}\}}\quad &\text{if $T\in \mathcal{T}^{M, (\ell)}$, $\ell\geqslant2$.}\end{cases}
$$
Then $\Psi$ is a surjective map from $\mathcal{T}^{M, (1)}$ onto $P^M_{0}\setminus\{0\}$, and for $\ell\geqslant2$,
$\Psi$ is a surjective map from $\mathcal{T}^{M, (\ell)}$ onto 
 $P^M_{\ell-1}\setminus P^M_{\ell-2}$.
 
Moreover, if $T$ is a surjective map from $\mathcal{T}^{(1)}$ onto $P_{0}\setminus\{0\}$, and for $\ell\geqslant2$,
$\Psi$ is a surjective map from $\mathcal{T}^{(\ell)}$ onto 
 $P_{\ell-1}\setminus P_{\ell-2}$.
\end{prop}
\begin{proof} We prove the claim by induction on $\ell\in \mathbb N$. For $\ell=1$, $\Psi$  is clearly a bijective map between $\mathcal{T}^{M, (1)}$ and $P_0^M\setminus\{0\}$.

Assume that there is $\ell\geqslant 2$ such that the claim is true for all $\ell'<\ell$ and 
let $T\in \mathcal{T}^{M, (\ell)}$. Choose $k\in\mathbb N$ and $T^{(i)}\in \bigcup_{j=1}^{\ell-1}\mathcal{T}^{M, (j)}$,  for $i=1,2,\ldots, k$, as in 
Proposition \ref{P:5.2} item (2). Denote, for each $i=1,2,\ldots, k$, the root of $T^{(i)}$ by $n_i$ such that $n_1<n_2<\cdots<n_k$, and $n_1$ is the root of $T$. By Condition (T3), we have $\{n_i: i=1,2,\ldots,k\}\in \mathcal F$.
Let $A_1=\{ i=1,2,\ldots, k:|T^{(i)}| =1\}$ 
and  $A_2=\{ i=1,2,\ldots, k:|T^{(i)}| >1\}$. 
We note that 
\begin{align*}
\Psi(T)&\ =\ \sum_{j=2}^\ell \theta^{j-1} 1^*_{\{t(j)\,:\, t\in T\text{ is a terminal node, and $|t|=j$}\}}\\
&\ =\ \theta\sum_{j=2}^\ell \theta^{j-2} 1^*_{\{t(j)\,:\, t\in T\text{ is a terminal node, and $|t|=j$}\}}\\
&\ =\ \theta\Big( \sum_{i\in A_1} e^*_{n_i}+ \sum_{i\in A_2} \sum_{j=2}^{\ell-1} \theta^{j-1}1^*_{\{t(j)\,:\, t\in T^{(i)}\text{ is a terminal node and $|t|=j$}\}}\Big)
\ =\ \theta \sum_{i=1}^k \Psi(T^{(i)}). 
\end{align*} 
It follows from the induction hypothesis that $\Psi(T^{(i)})\in P_{\ell-2}^M$, for $i=1,2,\ldots, k$,
and since 
$(\min \supp(\Psi(T^{(i)}))=\{n_1,n_2,\ldots, n_k\}\in \mathcal F$, we know that
$\Psi(T)\in P^{M}_{\ell-1}$. Since for at least one $i=1,2,\ldots$,  we have 
$T^{(i)}\in \mathcal T^{M, (\ell-1)}$, it follows that $\Psi(T)\in P^{M}_{\ell-1}\setminus P^{M}_{\ell-2}$.

In order to verify  that $\Psi: \mathcal T^{M, (\ell)} \rightarrow P^M_{\ell-1}\setminus P^M_{\ell-2}$
is surjective, let $f^*\in P^M_{\ell-1}\setminus P^M_{\ell-2}$ and write 
$f^*=\theta \sum_{i=1}^k f^*_i$, with $f^*_i\in P^M_{\ell-2}$ and $\{\mbox{supp}(f^*_i): i=1,2,\ldots, k\}$ is $\mathcal F$-allowable. Note that in the case $\ell >2$, for at least one
 $i=1,2,\ldots, k$ , we have that $f^*_i\not\in P^M_{\ell-3}$. By the inductive hypothesis, we can write 
 $f^*_i=\Psi(T^{(i)})$ for $i=1,2,\ldots, k$, where $T^{(i)}\in \mathcal \cup_{j=1}^{\ell-1} \mathcal{T}^{M, (j)}$. We can assume that $n_1<n_2<\cdots< n_k$, where $n_i$ is the root of $T^{(i)}$. We put
 $$T\ =\ \{n_1\} \cup\{(n_1,n_i, s)\,:\, i=1,2,\ldots, k, s\in [[\mathbb N]]^{<\omega}\cup\{\emptyset\}\text{ with }
 (n_i,s)\in T^{(i)}\}$$
and note that $T\in \mathcal{T}^{M, (\ell)}$ and $\Psi(T)=f^*$.
 
 The moreover part follows from the moreover part in Proposition \ref{P:5.2}.
\end{proof}

Finally, we state our key lemma which we shall prove by induction. We assume that $\mathcal F$ has the following property (by Theorem \ref{m1}, $\mathcal{S}_\xi$, $\xi<\omega_1$
has this property):
\begin{enumerate}
\item[($*$)]  If $A_1, \ldots, A_n$, $n\in \mathbb{N}$, are nonempty sets in $\mathcal{F}$ with 
$\min A_1 < \cdots < \min A_n$, then there exist nonempty sets $A_1' < \cdots < A_n'$ in $\mathcal{F}$ such that 
$\cup_{i=1}^n A_i = \cup_{i=1}^n A_i'$ and $\min A_i\leqslant \min A_i'$, for $i = 1, \ldots, n$. 
\end{enumerate} 

We denote the root of a tree $T$ by $r(T)$ and denote $\supp(r(T))$ by $\supp T$.  Recall that 
$$(\mathcal{F})^s \ :=\ \left\{\cup_{i=1}^s E_i\,:\, E_1 < \cdots < E_s, E_i\in \mathcal{S}\right\},$$
which, by Property ($*$), equals to the family
$$(\mathcal{F})_s \ :=\ \left\{\cup_{i=1}^s E_i\,:\, E_1, \ldots, E_s\mbox{ are pairwise disjoint}, E_i\in \mathcal{S}\right\}.$$

\begin{lem}\label{L:5.5} (The key lemma)  Let $T_1, T_2, \ldots, T_\ell$ be  $\mathcal{F}$-allowable trees, having pairwise disjoint supports, with $\{ r(T_i):i=1,2,\ldots,\ell\}\in (\mathcal{F})^s$ for some given $s\in \mathbb{N}$. Then there exist $\mathcal{F}[\mathcal{A}_2]$-admissible trees $T_1', T_2', \ldots, T'_{\ell'}$,  such that
\begin{enumerate}
\item $\supp T_1' < \supp T_2' < \cdots < \supp T'_{\ell'}$.
\item $\sum_{i=1}^\ell \Psi(T_i) = \sum_{i=1}^{\ell'} \Psi (T'_i)$.
\item $\ell' \leqslant \ell + k$, where $k = \#\{T_i: |T_i| = 1\}$. 
\item $\{r(T'_i): i=1,2,\ldots,\ell'\}\in (\mathcal{F}[\mathcal{A}_2])^s$. More precisely, if
 $\{r(T_i):i=1,\ldots,\ell\}$ is decomposed as  the union of $\mathcal{F}$ sets  $B_1 < B_2 < \cdots < B_s$, the set $\{ r(T'_i):i=1, \ldots,\ell'\}$ can be written as the union of  consecutive $\mathcal{F}[\mathcal{A}_2]$ sets $B'_1 < B'_2 < \cdots < B'_s$ such that $\min B_i \leqslant \min B_i'$ for $i = 1, \ldots, s$.

\end{enumerate}
\end{lem}

To see that Lemma \ref{L:5.5} implies Theorem \ref{mm2}, we simply let $\ell = 1$ and $T_1$ be an $\mathcal{F}$-allowable tree. If $|T_1| = 1$, then there is nothing to prove. If $|T_1|\geqslant 2$, the lemma gives exactly one $\mathcal{F}[\mathcal{A}_2]$-admissible tree $T_1'$ such that $\Psi(T_1') = \Psi(T_1)$, and we are done. To prove Lemma \ref{L:5.5}, we need the following easy observation.

\begin{prop}\label{P:5.6}
Let $k,m\in \mathbb N$, $a_1<a_2<\cdots<a_k$ in $\mathbb N$, $B_1<B_2< \cdots< B_m$ in $[\mathbb N]^{<\omega}$, 
have the following properties:
$\{a_1,a_2,\ldots, a_k\}\cap \bigcup_{j=1}^m B_j=\emptyset$ and 
$\{a_1,a_2,\ldots, a_k\}\cup\left\{\min B_j: j=1,2,\ldots, m\right\}\in (\mathcal F)^s$, for some $s\in \mathbb N$. Then the family 
$$
\mathcal{D}\ =\ \big(\{ (a_i,a_{i+1})\cap B_j\,:\, 0\leqslant i\leqslant k, 1\leqslant j \leqslant m\}\setminus\{\emptyset\}\big)\cup \{\{a_i\}\,:\, 1\leqslant i\leqslant k\},
$$
where $a_0=0$, and $a_{k+1}=\infty$  can be ordered into 
$(D_{j})_{j=1}^\ell$ with $D_1<D_2<\cdots<D_\ell$, $\ell\leqslant m+2k$, and
 $$\left\{\min D_j : j=1,2,\ldots, \ell\right\} \ \in\ (\mathcal{F}[{\mathcal A}_2])^s.$$
\end{prop}
\begin{proof}
\textbf{Ordering:} Let $D=(a_i,a_{i+1})\cap B_j\not=\emptyset$ and 
$D'=(a_{i'},a_{i'+1})\cap B_{j'}\not=\emptyset$. Then if $j<j'$, it follows that $D<D'$,
and if $j=j'$, and $i<i'$, then also $D<D'$.
Thus we can order $\mathcal D$ into $D_1<D_2<\ldots <D_l$.

\noindent\textbf{Size of $\mathcal{D}:$} By construction, it follows for each $D\in \mathcal D$, that 
either $D=\{a_i\}$, for some $i=1,2\ldots k$, or $\min D=\min B_j$, for some $j=1,2,\ldots, m$,
or there  is an $i=1,\ldots k$, and a $j=1,2,\ldots m$, for which $\min B_j <a_i<\min D\leqslant \max B_j$.
Since there can at most be $k$ elements in $\mathcal{D}$ of the third kind, it follows that $\ell\leqslant 2k+m$.

\noindent\textbf{From $\mathcal{F}$ to $\mathcal{F}[\mathcal{A}_2]$:}  Note that the 
elements of $\{\min D: D\in \mathcal{D}\}$, which are of the third kind, form a subset of a spread of 
$\{a_i:i=1,2, \ldots, k\}$, which implies that $\{\min D: D\in \mathcal{D}\} \in (\mathcal F[\mathcal{A}_2])^s$.
\end{proof} 
Let us give a short  overview of  the proof:
We will prove the lemma by induction on 
$$\max\text{length } \mathcal{C}\ :=\ \max\{|T|\,:\, T\in \mathcal{C}\}.$$
In the induction step, we decompose each $T_i$, which are not singletons, into trees of smaller lengths,
using Proposition \ref{P:5.2}. We apply the induction hypothesis to this, possibly much larger, collection of trees of smaller lengths.
 Then we ``glue'' the obtained new collection of $\mathcal{F}[\mathcal{A}]_2)$-admissible trees appropriately together. The possible existence 
 of singletons among the family $(T_i)_{i=1}^{\ell}$ may make it necessary to increase the number of trees so that the roots of $(T_i')_{i=1}^{\ell'}$ may
 only be in $(\mathcal{F}([\mathcal{A}]_2))^s$ instead of $(\mathcal F)^s$.

\begin{proof}[Proof of Lemma \ref{L:5.5}]
Let $\mathcal{C}$ be a finite collection of trees. We prove our claim by induction on the maximum length of trees in our collection.

Base case: $\max \text{length }\mathcal{C} = 1$. For $\ell, s\in \mathbb{N}$, suppose that $\mathcal{C}$ contains $\ell$ singleton trees $T_1, T_2,\ldots, T_\ell$ and $(r(T_i))_{i=1}^\ell\in (\mathcal{F})^s$. We simply set $T_i' = T_i$, for $i = 1, \ldots, \ell$. 

Inductive hypothesis: suppose that the lemma holds for all finite collections of trees $\mathcal{C}$ with $\max\text{length }\mathcal{C}\leqslant u$ for some $u\geqslant 1$. We shall show that it holds 
for all finite collections of trees $\mathcal{C}$
with $\max\text{length }\mathcal{C} = u+1$. 

Pick $\ell$ and $s\in \mathbb{N}$. Let $\mathcal{C}$ consist of $\ell$ disjoint $\mathcal{F}$-allowable trees $T_1, \ldots, T_\ell$ with $\max_{1\leqslant i\leqslant \ell} |T_i| = u+1$ and $(r(T_i))_{i=1}^s\in (\mathcal{F})^s$. Let $A_1 = \{n: |T_n| = 1\}$ and $A_2 = \{n: |T_n| > 1\}$,
and order $A_2$ into $ a_1 < a_2 < \cdots < a_{m}$. For each $n\in A_2$, let $T^{(i)}_{n}$, $i=1, \ldots, k_n$, be the trees defined as in Proposition \ref{P:5.2} item (2) for $T_n$.  For the collection of disjoint $\mathcal{F}$-allowable trees
$$\mathcal{C}'\ = \ \{T^{(i)}_n\,:\, n\in A_2, i = 1, \ldots, k_n\},$$
it follows that  $\max\text{length }\mathcal{C}' \leqslant u$ and $\{r(T): T\in \mathcal{C}'\}\in (\mathcal{F})_{m}$. By Property ($*$), $\{r(T): T\in \mathcal{C}'\}\in (\mathcal{F})^{m}$, and we can write 
$$\{r(T): T\in \mathcal{C}'\} \ =\ \bigcup_{i\in A_2} B_i, \quad B_{a_1} \ <\  \cdots\ <\ B_{a_m}, B_i\ \in\ \mathcal{F}, \min B_i\ \geqslant\ r\left(T_i\right).$$
Applying the inductive hypothesis to $\mathcal{C}'$, we obtain $\mathcal{F}[\mathcal{A}_2]$-admissible trees $\widetilde{T}_1, \widetilde{T}_2, \ldots, \widetilde{T}_v$ with $\supp \widetilde{T}_1< \supp \widetilde{T}_2 < \cdots < \supp \widetilde{T}_v$, $(r(\widetilde{T}_i))_{i=1}^v\in (\mathcal{F}[\mathcal{A}_2])^{m}$, and we can write 
$$(r(\widetilde{T}_i))_{i=1}^v \ =\ \bigcup_{i\in A_2}B'_i, \quad B'_{a_1} \ <\ \cdots\ <\ B'_{a_m}, B'_i\ \in\ \mathcal{F}[\mathcal{A}_2], \min B'_i\ \geqslant\ \min B_i.$$
For $n\in A_2$, define the following $\mathcal{F}[\mathcal{A}_2]$-admissible trees:
\begin{equation}\label{e201} \widehat T_n\ =\ \{(\min B'_n), (\min B'_n, t): t\in \widetilde{T}_i, r(\widetilde{T}_i)\in B'_n\}, \end{equation}
and for $n\in A_1$, put $\widehat T_n=T_n$.
 
Denote the new collection $\{\widehat T_n: n=1,2,\ldots, \ell\}$  of trees by $\mathcal{C}''$. From above, $\min B'_n\geqslant  r(T_n)$ for each $n\in A_2$; hence, the roots of the new trees in \eqref{e201} form a spread of $\{r(T_n): n\in A_2\}$. Therefore, $\{r(T): T\in \mathcal{C}''\}$ is a spread of $\{r(T): T\in \mathcal{C}\}$ and thus, is in $(\mathcal{F})^s$. Note that the set of singletons in $\mathcal{C}''$ is the same as in $\mathcal{C}'$. By our construction, if $\widehat{T}_i, \widehat{T}_j$ are not singletons with $r(\widehat{T}_i) < r(\widehat{T}_j)$, then $\supp \widehat{T}_i < \supp \widehat{T}_j$. Furthermore, since $\sum_{T\in \mathcal{C}'} \Psi(T) = \sum_{i=1}^v \Psi(\widetilde{T}_{i})$, we know that
$$\sum_{i=1}^\ell \Psi(\widehat{T}_i) \ =\ \sum_{i=1}^\ell\Psi (T_i).$$

We now apply Proposition \ref{P:5.6} to the sequence $(r(\widehat{T}_i))_{|\widehat{T}_i| = 1}$ and the collection of consecutive sets $\{\supp \widehat{T}_i: 1\leqslant i\leqslant \ell, |\widehat{T}_i| > 1\}$ to obtain sets $(D_i)_{i=1}^p$ satisfying $D_1 < D_2 < \cdots < D_p$, $(\min D_i)_{i=1}^p\in (\mathcal{F}[\mathcal{A}_2])^s$, and 
$$p \ \leqslant\ \#\{\widehat{T}_i\,:\, |\widehat{T}_i|>1\} + 2\cdot \#\{\widehat{T}_i\,:\, |\widehat{T}_i| = 1\}\ =\ \ell + |A_1|\ =\ \ell + k.$$

Now, we define $(T'_i)_{i=1}^p$ based on the following cases for each $D_i$, $1\leqslant i\leqslant p$. Order the set  $\{r(\widehat{T}_i):|\widehat{T}_i| = 1\}$ into $ \widehat{a}_1 < \widehat{a}_2 < \cdots < \widehat{a}_k$. For convenience, let $\widehat{a}_0 = 0$ and $\widehat{a}_{k+1} = \infty$. 

Case 1: if $D_i = r(\widehat{T}_j)$ for some $|\widehat{T}_j|=1$, we choose $T'_i = r(\widehat{T}_j)$. 

Case 2: if $D_i = (\widehat{a}_j, \widehat{a}_{j+1})\cap \supp \widehat{T}_q$ for some $0\leqslant j\leqslant k$ and for some $q$ with $|\widehat{T}_q|>1$, we choose an $\mathcal{F}[\mathcal{A}_2]$-admissible tree $T'_i$ for which $\Psi(T'_i)$ is the  projection of $\Psi(\widehat{T}_q)$ onto $(\widehat{a}_j, \widehat{a}_{j+1})$. This can be done thanks to Proposition \ref{P:5.3}.

The trees $(T'_i)_{i=1}^p$ are $\mathcal{F}[\mathcal{A}_2]$-admissible and obviously satisfy Properties (1), (2), and (3) of our claim. In order to  verify that $(T'_i)_{i=1}^p$ satisfies (4), we write
$$\{ r(T_i: i=1,\ldots,\ell\}\ =\ \bigcup_{i=1}^s E_i, \quad E_{1} \ <\  \cdots\ <\ E_{s}, E_i\ \in\ \mathcal{F}.$$
Since $r(\widehat{T}_i)\geqslant r(T_i)$, if we write 
$$\{r(\widehat{T}_i):i=1,\ldots,\ell\}\ =\ \bigcup_{i=1}^s E'_i, \quad E'_{1} \ <\  \cdots\ <\ E'_{s}, |E'_i| = |E_i|,$$
then $E'_i$ is a spread of $E_i$, and $E'_i\in \mathcal{F}$. As in the proof of Proposition \ref{P:5.6}, the set $(r(T'_i))_{i=1}^p = (\min D_i)_{i=1}^p$ is formed by adding to $(r(\widehat{T}_i))_{i=1}^{\ell}$ at most one integer between two consecutive integers in $(r(\widehat{T}_i))_{i=1}^{\ell}$. Therefore, if we write
$$\{r(T'_i): i=1,\ldots,\ell\}\ =\ \bigcup_{i=1}^s E''_i, \quad E''_{1} \ <\  \cdots\ <\ E''_{s},$$
where 
$$E''_i \ =\ [\min E'_i, \min E'_{i+1}-1]\cap \{r(T'_i)\,:\, 1\leqslant i\leqslant p\}\mbox{ and }\min E'_{s+1} = \infty,$$
then $E''_i$ is in $\mathcal{F}[\mathcal{A}_2]$, and $\min E''_i = \min E'_i\geqslant \min E_i$. 
This completes our proof. 
\end{proof}

From Lemma \ref{L:5.5} and Theorem \ref{mm1},  we obtain
\begin{cor} For any $\xi<\omega_1$, $0<\theta<1$, and $x\in c_{00}$,  it follows that 
$$\|x\|_{\mathcal{S}_\xi,\theta}\ \leqslant\ \|x\|_{\mathcal{S}_\xi,\theta, M} \ \leqslant\ \|x\|_{\mathcal{S}_\xi,\theta}.$$
\end{cor}

The following observation will be important in the next section.
\begin{cor}\label{C:5.8} Let $\xi<\omega_1$. Any collection of normalized vectors $(y_j)_{j=1}^k$ in $T[\mathcal{S}_\xi, \theta]$
for which $\big(\supp(y_j)\big)_{j=1}^k$ is $\mathcal{S}_\xi$-allowable, is $\frac{\theta}{3}$-equivalent  to the $\ell_1^k$- unit vector basis.
\end{cor} 

\begin{proof}
For $(a_j)_{j=1}^k\subset \mathbb{R}$, we have
$$\left\|\sum_{j=1}^k a_jy_j\right\|_{\mathcal{S}_\xi, \theta}\ \geqslant\ \frac{1}{3}\left\|\sum_{j=1}^k a_jy_j\right\|_{\mathcal{S}_\xi, \theta, M}\ \geqslant\ \frac{1}{3}\sum_{j=1}^k \|a_jy_j\|_{\mathcal{S}_\xi, \theta, M}\ =\ \frac{1}{3}\sum_{j=1}^k |a_j|.$$
\end{proof}

\section{Application: the cardinality  of the closed ideals of $\mathcal{L}(T[\mathcal S_\xi,\theta])$}\label{S:6}

We recall some notation. For a Banach space $X$, $\mathcal{L}(X)$ denotes the algebra of bounded linear operators on $X$. A { \em closed ideal in $\mathcal{L}(X)$} is a closed subspace $\mathcal J$ of $\mathcal{L}(X)$ (in the strong operator norm) which is closed under multiplication from the right and left by elements of  $\mathcal{L}(X)$. For  $ \mathcal{T}\subset \mathcal{L}(X)$,
let $\mathcal{J}^{\mathcal{T}}$ denote the closed ideal generated by  $\mathcal{T}$, i.e., 
$$\mathcal{J}^{\mathcal{T}}=\overline{\left\{ \sum_{j=1}^n A_j\circ T_j\circ B_j: n\in\mathbb{N},\, T_j\in\mathcal{T}, \, A_j, B_j\in \mathcal{L}(X), j=1,2,\ldots, n\right\}}.$$
For $T\in \mathcal{L}(X)$,  we write  $\mathcal{J}^T=\mathcal{J}^{\{T\}}$  and call $\mathcal{J}^T$ a {\em singly generated closed ideal}.

In \cite{SchlumprechtZsak18},  it was shown that ${\mathcal{L}}(\ell_p\oplus\ell_q)$, 
$1<p<q<\infty$, has $\mathfrak c$ many singly generated closed subideals, which is, for trivial set theoretical reasons, the largest possible. In \cite{FreemanSschlumprechtZsak17}, this result was extended to the spaces $\ell_p\oplus c_0$, $\ell_p\oplus \ell_\infty$, and $\ell_1\oplus \ell_p$, $1<p<\infty$.
It might be worth mentioning here that the case $\ell_1\oplus c_0$ is still open. It is not even known whether or not
$\mathcal{L}(\ell_1\oplus c_0)$ has infinite many closed subideals.
Johnson and Schechtman \cite{JohnsonSchechtman21} proved that the space $\mathcal{L}(L_p)$, $1<p<\infty$,
has even $2^{\mathfrak c}$  closed subideals, which is also, for set-theoretical reason, the largest possible.
Using an argument from \cite{JohnsonSchechtman21}, it was also possible
to show that  $\mathcal L(X)$ has $2^{\mathfrak c}$ many closed subideals, where $X=\ell_p\oplus \ell_q,\ell_p\oplus c_0, \ell_1\oplus \ell_p$, $\ell_p\oplus \ell_\infty$, for $1<p<q<\infty$ \cite{FreemanSchlumprechtZsak21}.
The  argument from \cite{JohnsonSchechtman21}, which we will also use, is as follows:
\begin{prop}\label{P:6.0} Assume that $(T_N: N\in[\mathbb N]^\omega)$ is a uniformly bounded family of operators on a Banach space $X$ which satisfies the following {\em tree condition }for some $c>0$:
\begin{equation} T_M\ \in\ \mathcal{J}^{T_{N}} \text{ if $M\subset N$, and } \mbox{dist}(T_M,\mathcal{J}^{T_N})\ \geqslant\ c\text{ if } M\setminus N\in [\mathbb N]^{\omega}.  \tag{T}
\end{equation}
Let $\mathcal{C}\subset [\mathbb N]^{\omega} $ be of cardinality $\mathfrak c$, consisting of pairwise almost disjoint sets
 (meaning  that $M\cap N$ is finite, for all $M,N\in \mathcal C$, $M\not=N$). Then $(\mathcal{J}^{\{T_N: N\in\mathcal A\}}: \mathcal {A}\subset \mathcal{C})$ is a family of pairwise distinct closed subideals. 
\end{prop}

In unpublished work, Johnson \cite{Johnsonpersonal2020} showed that the space of operators on $T_1=T[\mathcal S_{1}, \theta]$ has $2^{\mathfrak c}$ 
closed ideals. 
We will here reproduce his arguments, and based on the results shown in the previous sections, we will deduce that his arguments extend to  $T[\mathcal S_{\xi}, \theta]$, for all $\xi<\omega_1$. We start by recalling a result by Johnson, which is based on \cite[Theorem A]{PelzynskiRosenthal74/75}.
 \begin{prop}{\rm \cite[Proposition 1]{Johnson76}}\label{P:6.1} For each $m\in \mathbb N$, there is a number  $L(m)\in \mathbb N$
 (the proof shows that one can take $L(m)=2^{2m^6}$) so that for every Banach space $X$ with a $ 1$ unconditional 
 basis $(e_j)$, the  following holds:
 
 For each  $m$-dimensional subspace $E$  of $X$, there exists an automorphism
 $S: X\to X$, with $\|S\|\cdot\|S^{-1}\|\leqslant2$, and pairwise disjoint supported vectors
 $(x_j)_{j=1}^{L(m)}$ in $X$ so that $S(E)\subset \text{span}(x_j: j=1,2,\ldots,L(m) )$.
 \end{prop}
 \begin{cor}\label{C:6.2} Let  $\xi<\omega_1$, $0<\theta<1$, $X=T[\mathcal{S}_\xi, \theta]$, and $(e_j)$ be its unit basis.
 For any $m\in \mathbb N$ and any $m$-dimensional subspace $E$ of $X_{L(m)}=\overline{\text{span}(e_j: j\geqslant L(m))}$  (where $L(m)$ is as in
 Proposition  \ref{P:6.1}), there exists an $L(m)$-dimensional subspace $F$  of $X_{L(m)}$
  generated by a sequence  $(x_j)_{j=1}^{L(m)}$ of disjointly supported vectors, which is $\frac{\theta}{3}$ equivalent to the 
  $\ell_1^{L(m)}$-vector basis, and there exists an automorphism $S: X_{L(m)}\to X_{L(m)}$, with 
  $\|S^{-1}\|\cdot \|S\|\leqslant2$ and $S(E)\subset F$.
 \end{cor} 
 \begin{proof} We apply Proposition  \ref{P:6.1} to the space $X_{L(m)}$. Then our claim follows from Corollary  \ref{C:5.8} and our previous observation  
  that $\mathcal{S}_1\subset \mathcal{S_\xi}$.
 \end{proof}
 We are now ready to state the main result of this section.
\begin{thm} {\rm  \cite{Johnsonpersonal2020} } Assume that $X$ is a Banach space with a $1$-unconditional basis $(e_j)$ 
which is not equivalent to the unit vector basis of $\ell_1$ and satisfies the following condition for some $c>0$ and some sequence $(L(m))_{m=1}^\infty\subset \mathbb N $:

For any $m\in \mathbb N$ and any $m$-dimensional subspace $E$ of $X_{L(m)}=\overline{\text{span}(e_j: j\geqslant L(m))}$,  there exists a finite-dimensional subspace $F$  of $X_{L(m)}$
  generated by a basis  $(f_j)_{j=1}^{L(m)}$, which is $c$-equivalent to the 
  $\ell_1^{L(m)}$-unit basis, and there exists an automorphism $S: X_{L(m)}\to X_{L(m)}$, with 
  $\|S^{-1}\|\cdot \|S\|\leqslant2$ and $S(E)\subset F$.

 Then there exists a family $(T_N: N\in [\mathbb N]^\omega)$
of spatial projections ($T_N$ is a projection on some subspace of $X$ generated by a subsequence of $(e_j)$), which satisfies 
(T).
\end{thm}
 
 \begin{proof} Using on the one hand that $(e_j)$ is not equivalent to the unit vector basis of $ \ell_1$,
 on the other hand the sequence $(L(m))$, we 
 choose recursively  natural  numbers $m_1<n_1< m_2< n_2<\cdots $ as follows:
 
 We let $m_1=1$, and assuming that $m_k\in \mathbb N$ has been chosen, we choose $n_k>m_k$ so that 
 there exist numbers  $(p_j^{(k)})_{j=m_k}^{n_k}\subset [0,1]$ so that 
 $\sum_{j=m_k}^{n_k} p_j^{(k)}=1$ and 
 \begin{equation}\left\|\sum_{j=m_k}^{n_k} p^{(k)}_j e_j\right\|\ <\ \frac1{k m_{k}},
 \end{equation} 
 and assuming $n_k\in \mathbb N$ has been chosen, we choose $m_{k+1}=L(n_k)$.
 
For $N\in[\mathbb N]^{\omega}$, we define  $T_N$ to be the projection onto the subspace of $X$ generated by $(e_i: i\in \bigcup_{j\in N} [m_j,n_j])$
 and claim that the so defined  family $(T_N:N\in [\mathbb N]^{\omega})$ satisfies (T).
 The first condition in (T) is trivially satisfied. To verify the second condition, we define for $k\in\mathbb N$ the following functional $\Psi_k\in \mathcal{L}^*(X)$:
 $$\Psi_k(S)\ =\ \sum_{j=m_k}^{n_k} p^{(k)}_j e^*_j(S(e_j)), \text{ for }S\in \mathcal{L}(X).$$
 Clearly,  $\|\Psi_k\|_{\mathcal{L}^*(X)}=1$, and if  $N\in [\mathbb N]^{\omega}$ and $k\in N$, then $\Psi_k(T_N)=1$.
 It is therefore enough to show that for $A,B\in \mathcal{L}(X)$, $\|A\|,\|B\|\leqslant1$, and $M,N\in [\mathbb N]^{\omega}$, with 
 $M\setminus N\in[\mathbb N]^{\omega}$, we have 
 \begin{equation}\label{E:6.2}\lim_{k\to \infty, k\in M\setminus N} \Psi_k(A\circ T_N\circ B)\ =\ 0\end{equation}
 We compute for $k\not\in N$ that
 \begin{align*} \big| \Psi_k(A\circ T_N\circ B)\big|&\ \leqslant\ \sum_{j=m_k}^{n_k} p^{(k)}_j  | e^*_j (AT_NB)(e_j)|\\
 &\ \leqslant\ \sum_{j=m_k}^{n_k} p^{(k)}_j  \|T_NB(e_j)\|\\
 &\ \leqslant\ \underbrace{ \sum_{s=1}^{n_{k-1}} \sum_{j=m_k}^{n_k} p^{(k)}_j  |B_s(e_j)|}_{=\ \Sigma_1}+ 
 \underbrace{\sum_{j=m_k}^{n_k} p^{(k)}_j  \| P\circ B(e_j)\|}_{=\ \Sigma_2},
 \end{align*} 
 where $B_s= e^*_s\circ B\in X^*$, with $\| B_s\|_{X^*}\leqslant1$, for $s\in\mathbb N$, and 
 $P: X\to X_{m_{k+1}}$ is the basis  projection onto $X_{m_{k+1}} =\overline{\mbox{span}(e_j: j\geqslant m_{k+1})}$.
 
To estimate $\Sigma_1$, we observe that 
\begin{align}\label{E:6.3}\Sigma_1\ =\ \sum_{s=1}^{n_{k-1}} B_s\left(\sum_{j=m_{k}}^{n_{k}} p^{(k)}_j \mbox{sign}(B_s(e_i))e_j\right)
\ \leqslant\ \sum_{s=1}^{n_{k-1}} \left\|\sum_{j=m_{k}}^{n_{k}} p^{(k)}_j \mbox{sign}(B_s(e_i))e_j\right\|\ \leqslant\ \frac{1}{k}.
\end{align}

 To estimate $\Sigma_2$, we consider the  space $E=P\circ B(\mbox{span}(e_j: m_k\leqslant j\leqslant n_k))$
which is a subspace of $X_{m_{k+1}}=X_{L(n_k)}$  and has dimension not larger than $n_k$.
 Therefore, we  can find a  finite-dimensional subspace $F$ of $X_{L(n_k)}$, with a basis $(f_{j})_{j=1}^{\mbox{dim} F}$ which is $c$-equivalent to the $\ell_1^{\mbox{dim} F}$-unit basis, and an automorphism $S$ on  $X_{m_{k+1}}$,
for which $S(E)\subset F$, and $\|S\|\cdot\|S^{-1}\|\leqslant2$. We have
\begin{align*}\Sigma_2&\ \leqslant\ 2 \sum_{j=m_k}^{n_k} p^{(k)}_j \|SPB(e_j)\|\\
&\ \leqslant\ 2\sum_{j=m_k}^{n_k} p^{(k)}_j \sum_{s=1}^{\dim(F)} |\widetilde B_s(e_j)|
\intertext{where $\widetilde B_s(x)=f^*_s(SPB(x))$, for $s=1,2, \ldots, \mbox{dim} F$, $x\in X$, $(f^*_j)_{j=1}^{\mbox{dim} F}$ being the 
coordinate functionals in $F$,}
&\ =\ 2\sum_{s=1}^{\dim(F)} \sum_{j=m_k}^{n_k} p^{(k)}_j |\widetilde B_s(e_j)|\\
&\ =\ 2\sum_{s=1}^{\dim(F)} \widetilde B_s\left( \sum_{j=m_k}^{n_k} p^{(k)}_j \mbox{sign}(\widetilde B_s(e_j))e_j\right)\\
&\ \leqslant\ 2c \left\|\sum_{s=1}^{\dim(F)} \widetilde B_s\left( \sum_{j=m_k}^{n_k} p^{(k)}_j \mbox{sign}(\widetilde B_s(e_j))e_j\right)f_s\right\|_F\\
&\ =\ 2c \left\| SPB\left(\sum_{j=m_k}^{n_k} p^{(k)}_j \mbox{sign}(\widetilde B_s(e_j))e_j\right)\right\|_F\\
&\ \leqslant\ 4c\left\|\sum_{j=m_k}^{n_k} p^{(k)}_j e_j\right\|\ \leqslant\ \frac{ 4c}{k}.
\end{align*} 
This, together with \eqref{E:6.3}, implies our claim \eqref{E:6.2}.
 \end{proof}

\end{document}